\documentclass{jt-calcs}
\usepackage{array}
\usepackage{booktabs}
\usepackage{mathrsfs}

\newcommand{\bl}[1]{\ensuremath{\mathscr{#1}}}

\renewcommand{\cref}{\Cref}

\setlist[enumerate,1]{label=(\roman*), ref=(\roman*)}

\title{On Unipotent Supports of Reductive Groups with a Disconnected Centre}
\author{Jay Taylor}
\address{FB Mathematik, TU Kaiserslautern, Postfach 3049, 67653 Kaiserslautern, Germany}
\email{taylor@mathematik.uni-kl.de}

\begin{document}
\start
\begin{abstract}
Let $\bG$ be a simple algebraic group defined over a finite field of good characteristic, with associated Frobenius endomorphism $F$. In this article we extend an observation of Lusztig, (which gives a numerical relationship between an ordinary character of $\bG^F$ and its unipotent support), to the case where $Z(\bG)$ is disconnected. We then use this observation in some applications to the ordinary character theory of $\bG^F$.
\end{abstract}

Throughout this article $\bG$ will be a connected reductive algebraic group over $\overline{\mathbb{F}}_p$, an algebraic closure of the finite field $\mathbb{F}_p$, where $p$ is a good prime for $\bG$. Furthermore, we assume $\bG$ is defined over $\mathbb{F}_q \subset \overline{\mathbb{F}}_p$ (where $q$ is a power of a $p$) and $F:\bG\to \bG$ is the associated Frobenius endomorphism. Throughout we will denote an algebraic group in bold and its corresponding rational structure in roman, for instance $G := \bG^F$.

\section{Introduction}
One of the major open problems in the representation theory of finite groups is the determination of the full ordinary character table of $G$. Arguably one of the most difficult parts in computing the character table of $G$ comes from determining the character values at unipotent elements. To handle this problem Lusztig has developed a geometric theory of character sheaves, which gives the framework to express links between the geometry of $\bG$ and the character theory of $G$.

Let $\Irr(G)$ denote the set of ordinary irreducible characters of $G$. In \cite{lusztig:1992:a-unipotent-support} Lusztig has associated to every $\chi \in \Irr(G)$ a unique $F$-stable unipotent class of $\bG$ called the \emph{unipotent support} of $\chi$, which we denote by $\mathcal{O}_{\chi}$. This was originally done under the assumption that $p$ and $q$ are sufficiently large, however the assumptions on $p$ and $q$ were later removed in \cite{geck-malle:2000:existence-of-a-unipotent-support}. A consequence of the existence of unipotent supports is that we have a surjective map
\begin{equation*}
\Phi_G : \Irr(G) \to \{F\text{-stable unipotent conjugacy classes of }\bG\}
\end{equation*}
given by $\Phi_G(\chi) := \mathcal{O}_{\chi}$.

Recall that Lusztig has shown for each $\chi \in \Irr(G)$ there is a well defined integer $n_{\chi}$ such that $n_{\chi}\cdot \chi(1)$ is a polynomial in $q$ with integer coefficients. If $x \in \bG$ then we write $A_{\bG}(x)$ for the component group $C_{\bG}(x)/C_{\bG}(x)^{\circ}$. The following result was observed by Lusztig in \cite[13.4.4]{lusztig:1984:characters-of-reductive-groups} and later verified by H\'{e}zard and Lusztig through a detailed case by case analysis, (see \cite{lusztig:2009:unipotent-classes-and-special-Weyl} and \cite{hezard:2004:thesis}).

\begin{thm}[H\'{e}zard, Lusztig]\label{thm:lusztig-hezard}
Assume $\bG/Z(\bG)$ is simple, $Z(\bG)$ is connected and $\mathcal{O}$ is an $F$-stable unipotent class of $\bG$. There exists a character $\chi \in \Irr(G)$ such that $\Phi_G(\chi)=\mathcal{O}$ and $n_{\chi} = |A_{\bG}(u)|$ for any $u \in \mathcal{O}$.
\end{thm}

\noindent The usefulness of characters satisfying this technical condition can be seen in Kawanaka's theory of \emph{generalised Gelfand--Graev representations} (GGGRs). In particular, if $\chi$ satisfies this technical condition then it severly limits the appearance of $\chi$ as a constituent of the character of a GGGR associated to a unipotent element in $\mathcal{O}_{\chi}^F$.

The main result of this article is an extension of \cref{thm:lusztig-hezard} to include all simple groups, see \cref{thm:A}. The layout of this paper is as follows. In \cref{sec:comp-groups} we will show there exist class representatives of each $F$-stable unipotent class such that a sensible generalisation of \cref{thm:lusztig-hezard} can be made, (we call such representatives \emph{well chosen}). In \cref{sec:main-result} we will state and prove our main result. Using this result we prove in Sections \cref{sec:GGGR} to \cref{sec:conjecture}, (under the assumption $p$ and $q$ are large), that a conjecture of Kawanaka on unipotently supported class functions holds for simple groups which are not spin / half-spin groups. This is based on the work of Geck and H\'{e}zard in \cite{geck-hezard:2008:unipotent-support} which uses \cref{thm:lusztig-hezard}. Finally, in Sections \cref{sec:fourth-roots} and \cref{sec:orth-symp-groups} we give an expression of a certain fourth root of unity that is associated to GGGRs when $\bG$ is a special orthogonal or symplectic group. This is based on an argument used by Geck in \cite{geck:1999:character-sheaves-and-GGGRs} which uses \cref{thm:lusztig-hezard}.

\begin{acknowledgments}
The work presented in this article was completed as part of the authors PhD qualification under the supervision of Prof.\ Meinolf Geck. The author wishes to express his deepest thanks to Prof.\ Geck for proposing this problem and for his vital guidance. Part of this work was completed while the author spent six months at Lehrstul D f\"{u}r Mathematik, RWTH Aachen University as an ERASMUS exchange student and also while the author spent two weeks at Universit\'{e} Montpellier 2. The author warmly thanks both institutions for their kind hospitality during his visits. The author also thanks Dr.\ Frank L\"{u}beck for many essential discussions, (in particular for suggesting the approach taken in section 2 of this paper), and Prof.\ C\'{e}dric Bonnaf\'{e}, (for assistance with some technical issues in type $\A$).
\end{acknowledgments}

\section{Component Groups}\label{sec:comp-groups}
For the remainder of this article, unless otherwise stated, we will assume $\bG$ is simple. We fix a regular embedding $\iota : \bG \to \btG$, (as defined in \cite[\S 7]{lusztig:1988:reductive-groups-with-a-disconnected-centre}), where $\btG$ is a connected reductive group such that $Z(\btG)$ is connected. We now fix a dual group of $\bG$, (resp.\ $\btG$), which we denote by $\bG^{\star}$, (resp.\ $\btG^{\star}$) and a Frobenius endomorphism $F^{\star} : \bG^{\star} \to \bG^{\star}$ such that the pairs $(\bG,F)$ and $(\bG^{\star},F^{\star})$ are in duality. Similarly we denote again by $F$ and $F^{\star}$ compatible Frobenius endomorphisms of $\bG$ and $\bG^{\star}$. Recall that $\iota$ determines a corresponding surjective morphism $\iota^{\star} : \btG^{\star} \to \bG^{\star}$ between dual groups, which is defined over $\mathbb{F}_q$. For smoothness of the exposition if $Z(\bG)$ is connected then we will take $\btG$ to be $\bG$ and $\iota$ to be such that $\iota(g) = g$ for all $g \in \bG$.

The embedding $\iota$ defines a bijection between the unipotent conjugacy classes of $\bG$ and the corresponding classes in $\btG$. If $u$ is a unipotent element of $\bG$ then we will implicitly identify $u$ with its image $:= \iota(u) \in \btG$. The purpose of this section is to show that in a simple algebraic group we can always pick representatives of an $F$-stable unipotent class such that they are well chosen. By this we mean the following.

\begin{definition}
Let $\mathcal{O}$ be an $F$-stable unipotent class of $\bG$. We say a class representative $u \in \mathcal{O}^F$ is \emph{well chosen} if $|A_{\bG}(u)^F| = |Z_{\bG}(u)^F||A_{\btG}(u)|$, where $Z_{\bG}(u)$ is the image of $Z(\bG)$ in $A_{\bG}(u)$.
\end{definition}

\noindent We will start by showing that well-chosen class representatives always exist when $\bG$ is adjoint.

\begin{lem}\label{lem:abelian-comp-group}
Let $\bG$ be a connected reductive algebraic group and $\mathcal{C}$ an $F$-stable conjugacy class of $\bG$ such that $A_{\bG}(x)$ is abelian for some $x \in \mathcal{C}^F$. Then the following are equivalent:
\begin{enumerate}
	\item the number of $G$-conjugacy classes contained in $\mathcal{C}^F$ is $|A_{\bG}(x)|$\label{pt:i},
	\item $A_{\bG}(x)^F = A_{\bG}(x)$\label{pt:ii},
	\item $A_{\bG}(x')^F = A_{\bG}(x')$ for all $x' \in \mathcal{C}^F$\label{pt:iii}.
\end{enumerate}
\end{lem}

\begin{proof}
By \cite[Theorem 4.3.5]{geck:2003:intro-to-algebraic-geometry} we know the $G$-classes contained in $\mathcal{C}^F$ are in bijective correspondence with the $F$-conjugacy classes of $A_{\bG}(x)$. Therefore the number of $G$-classes contained in $\mathcal{C}^F$ is the same as $|A_{\bG}(x)|$ if and only if every $F$-conjugacy class of $A_{\bG}(x)$ contains only one element. As $A_{\bG}(x)$ is abelian this proves the equivalence of the first two statements. If \ref{pt:i} is true for $x \in \mathcal{C}^F$ then it is true for all $x \in \mathcal{C}^F$ as $|A_{\bG}(x)|$ is independent of the choice of $x$. Therefore this together with the equivalence of \ref{pt:i} and \ref{pt:ii} gives the equivalence of \ref{pt:i} and \ref{pt:iii}.
\end{proof}

\begin{lem}\label{lem:inner-auto-comp-grp}
Let $\bG$ be a connected reductive algebraic group and $\mathcal{C}$ an $F$-stable conjugacy class of $\bG$. If the automorphism of $A_{\bG}(x)$ induced by $F$, for some $x \in \mathcal{C}^F$, is an inner automorphism then there exists $x' \in \mathcal{C}^F$ such that $A_{\bG}(x')^F = A_{\bG}(x')$.
\end{lem}

\begin{proof}
Assume there exists an element $y \in C_{\bG}(x)$ such that for all $zC_{\bG}(x)^{\circ} \in A_{\bG}(x)$ we have\linebreak $F(zC_{\bG}(x)^{\circ}) = y^{-1}zyC_{\bG}(x)^{\circ}$. As $y \in \bG$, by the Lang-Steinberg theorem, there exists $g \in \bG$ such that $y = g^{-1}F(g) \Rightarrow F(g) = gy$. Clearly $A_{\bG}({}^gx) = {}^gA_{\bG}(x)$ so $gzC_{\bG}(x)^{\circ}g^{-1} \in A_{\bG}({}^gx)$. Furthermore we have
\begin{equation*}
F(gzC_{\bG}(x)^{\circ}g^{-1}) = gy(y^{-1}zy)C_{\bG}(x)^{\circ}y^{-1}g^{-1} = gzC_{\bG}(x)^{\circ}g^{-1},
\end{equation*}
so $F$ acts trivially on $A_{\bG}({}^gx)$. Taking $x' = {}^gx \in \mathcal{C}^F$ we have $A_{\bG}(x')^F = A_{\bG}(x')$ as required. Note $x'$ is fixed by $F$ because $y \in C_{\bG}(x)$.
\end{proof}

\begin{prop}\label{thm:triv-comp-action-adj}
Let $\bG$ be a connected reductive algebraic group with connected centre such that $\bG/Z(\bG)$ is simple, then for any $F$-stable unipotent class $\mathcal{O} \subset \bG$ we have $A_{\bG}(u)^F = A_{\bG}(u)$ for some unipotent element $u \in \mathcal{O}^F$.
\end{prop}

\begin{proof}
The natural surjective morphism of algebraic groups $\pi : \bG \to \bG/Z(\bG)$ induces a bijection between the unipotent classes of $\bG$ and $\bG/Z(\bG)$. If $u \in \bG$ is unipotent then the restriction of $\pi$ to $C_{\bG}(u)$ induces an isomorphism $A_{\bG}(u) \cong A_{\bG/Z(\bG)}(\pi(u))$, which is defined over $\mathbb{F}_q$. Therefore we may prove this result using any group $\bG$ which has a connected centre and simple quotient $\bG/Z(\bG)$.

The structure of the component groups $A_{\bG}(u)$ have been determined on a case by case basis. In the case where $\bG$ is an adjoint exceptional group it is known that $A_{\bG}(u)$ is either trivial or isomorphic to a symmetric group $\mathfrak{S}_2$, $\mathfrak{S}_3$, $\mathfrak{S}_4$ or $\mathfrak{S}_5$, (see for example the tables in \cite[\S 13.1]{carter:1993:finite-groups-of-lie-type}). Every automorphism of such a group is an inner automorphism, so the result holds by \cref{lem:inner-auto-comp-grp}.

We now turn to the classical groups. The case of type $\A_n$ is trivial as all component groups are trivial. Assume $\bar{\bG}$ is either a symplectic or special orthogonal group and $\bar{\mathcal{O}} \subset \bar{\bG}$ is an $F$-stable unipotent class with representative $\bar{u} \in \bar{\mathcal{O}}^F$. The component group $A_{\bar{\bG}}(\bar{u})$ is an elementary abelian 2-group whose order is given in \cite[IV - 2.26 and 2.27]{springer-steinberg:1970:conjugacy-classes}. By inspecting the descriptions of $|A_{\bar{\bG}}(\bar{u})|$ we see by \cite[p.\ 38]{wall:1963:conjugacy-classes-classical-groups}, (resp.\ \cite[p.\ 42]{wall:1963:conjugacy-classes-classical-groups}), for the symplectic groups, (resp.\ special orthogonal groups), that the number of $\bar{G}$-classes contained in $\bar{\mathcal{O}}^F$ is the same as $|A_{\bar{\bG}}(\bar{u})|$. Note the statement given in \cite{wall:1963:conjugacy-classes-classical-groups} is for the orthogonal groups but the maximal number of $\bar{G}$-classes contained in $\bar{\mathcal{O}}^F$ is $|A_{\bar{\bG}}(\bar{u})|$ and there can only be less classes in the orthogonal group so this suffices. As $A_{\bar{\bG}}(\bar{u})$ is abelian we have $A_{\bar{\bG}}(\bar{u})^F = A_{\bar{\bG}}(\bar{u})$ by \cref{lem:abelian-comp-group}. The odd dimensional special orthogonal groups are adjoint groups of type $\B_n$, so this case is covered.

Let $\bG$ be an adjoint group of type $\C_n$ or $\D_n$ then there exists a symplectic or special orthogonal group $\bar{\bG}$ such that we have an isogeny $\pi : \bar{\bG} \to \bG$ which is defined over $\mathbb{F}_q$. Let $\bar{u}$ be a unipotent element of $\bar{\bG}$ and set $u:=\pi(\bar{u})$. The restriction of $\pi$ to $C_{\bar{\bG}}(\bar{u})$ induces a surjective morphism $A_{\bar{\bG}}(\bar{u}) \to A_{\bG}(u)$ defined over $\mathbb{F}_q$, therefore $A_{\bar{\bG}}(\bar{u})^F = \A_{\bar{\bG}}(\bar{u})$ implies $A_{\bG}(u)^F = A_{\bG}(u)$ so we are done.
\end{proof}

\begin{rem}
The statement of \cref{thm:triv-comp-action-adj} is not a new result. It follows in most cases from a much stronger statement concerning the existence of so called split elements, see \cite[Remark 5.1]{shoji:1987:green-functions-of-reductive-groups}. However, we include this here as it gives a simpler argument for this statement and also circumvents the caveat that split elements do not always exist in type $\E_8$.
\end{rem}

We want to discuss how to relate the component group $A_{\bG}(u)$ to the component group $A_{\btG}(u)$, (for this we follow \cite[p.\ 306]{geck:1996:on-the-average-values}). The group $\btG$ is an almost direct product $\btG = \bG Z(\btG)$ so we get an almost direct product $C_{\btG}(u) = C_{\bG}(u)Z(\btG)$ and the restriction of $\iota$ to the centraliser $C_{\bG}(u)$ induces a surjective map $A_{\bG}(u) \to A_{\btG}(u)$. It is easy to see that the kernel of this map is the image of $Z(\bG)$ in $A_{\bG}(u)$ so we obtain a sequence of maps
\begin{equation}\label{eq:short-exact-alg}
Z(\bG) \to A_{\bG}(u) \to A_{\btG}(u),
\end{equation}
which is exact at $A_{\bG}(u)$. As all morphisms are defined over $\mathbb{F}_q$ this induces a sequence
\begin{equation}\label{eq:short-exact-frob}
Z(\bG)^F \to A_{\bG}(u)^F \to A_{\btG}(u),
\end{equation}
which is again exact at $A_{\bG}(u)^F$, (where here we assume $u$ is chosen such that $A_{\btG}(u)^F = A_{\btG}(u)$). Although the map $A_{\bG}(u) \to A_{\btG}(u)$ in \eqref{eq:short-exact-alg} is surjective it is not necessarily the case that the map $A_{\bG}(u)^F \to A_{\btG}(u)$ in \eqref{eq:short-exact-frob} is surjective. If we can show that this is the case then the existence of well-chosen elements will be assured.

There are two cases where this is trivial to show. If $|A_{\btG}(u)| = 1$ then it is obvious that such a map is surjective. The second case is when $|A_{\bG}(u)| = |A_{\btG}(u)|$ because then we must have $A_{\bG}(u) \cong A_{\btG}(u)$, which implies $A_{\bG}(u)^F \cong A_{\btG}(u)^F$. In fact we are almost always in the situation of these two trivial cases, (see for example the information gathered in \cite[Chapter 2]{taylor:2012:thesis}). We deal with the remaining cases in exceptional groups in two lemmas. In the following lemmas we label the unipotent conjugacy classes as in \cite[\S 13.1]{carter:1993:finite-groups-of-lie-type}.

\begin{lem}\label{lem:comps-E6}
Let $\bG$ be the simply connected group of type $\E_6$ and $\mathcal{O}$ the unipotent class $\E_6(a_3)$. For any $u \in \mathcal{O}^F$ we have $A_{\btG}(u)^F = A_{\btG}(u)$ and $A_{\bG}(u)^F \to A_{\btG}(u)$ is surjective.
\end{lem}

\begin{proof}
We see that $A_{\bG}(u)$ is isomorphic to $\mathbb{Z}_3 \times \mathbb{Z}_2$, as it is a finite group of order 6 with a non-trivial centre. The group $A_{\btG}(u)$ is abelian and $A_{\btG}(u)^F = A_{\btG}(u)$ by \cref{thm:triv-comp-action-adj} and \cref{lem:abelian-comp-group}. Let $Z(\bG) = \langle z \rangle$ and $x \in A_{\bG}(u)$ be the element of order 2 such that $A_{\bG}(u) = \langle z,x \rangle$. As the action of $F$ on $A_{\btG}(u)$ is trivial we must have $F(x) \in \{x,xz,xz^2\}$. However the element $x$ is of order 2 but $xz$ and $xz^2$ are of order 6, so $F(x) = x$.
\end{proof}

\begin{lem}\label{lem:comps-E7}
Let $\bG$ be the simply connected group of type $\E_7$ and $\mathcal{O}$ the unipotent class $\D_4(a_1)+\A_1$, $\E_7(a_3)$, $\E_7(a_4)$ or $\E_7(a_5)$. Then for some $u \in \mathcal{O}^F$ we have $A_{\btG}(u)^F = A_{\btG}(u)$ and $A_{\bG}(u)^F \to A_{\btG}(u)$ is surjective.
\end{lem}

\begin{proof}
Let $u \in \mathcal{O}^F$ be an element such that its image $u \in \mathcal{\tilde{O}}^F$ has the property $A_{\btG}(u)^F = A_{\btG}(u)$, which is possible because of \cref{thm:triv-comp-action-adj}. We use the work of Mizuno \cite{mizuno:1980:conjugacy-classes-of-E7-and-E8} to confirm that $F$ only acts non-trivially on $Z_{\bG}(u)$. By
\begin{enumerate}
	\item \cite[Lemma 30]{mizuno:1980:conjugacy-classes-of-E7-and-E8} for $\D_4(a_1)+\A_1$ -- Mizuno label $\D_4(a_1)+\A_1$,
	\item \cite[Lemma 13]{mizuno:1980:conjugacy-classes-of-E7-and-E8} for $\E_7(a_3)$ -- Mizuno label $\D_6+\A_1$,
	\item \cite[Lemma 17]{mizuno:1980:conjugacy-classes-of-E7-and-E8} for $\E_7(a_4)$ -- Mizuno label $\D_6(a_1)+\A_1$,
	\item and \cite[Lemma 21(2)]{mizuno:1980:conjugacy-classes-of-E7-and-E8} for $\E_7(a_5)$ -- Mizuno label $\D_6(a_2)+\A_1$,
\end{enumerate}
we know the number of $G$-classes contained in $\mathcal{O}^F$ is $|A_{\bG}(u)|$. If $\mathcal{O}$ is not $\E_7(a_5)$ then the group $A_{\bG}(u)$ is abelian so the result follows by \cref{lem:abelian-comp-group}.

Assume $\mathcal{O}$ is the class $\E_7(a_5)$, by \cite[Table 9]{mizuno:1980:conjugacy-classes-of-E7-and-E8} we see $A_{\bG}(u)$ is isomorphic to $\mathfrak{S}_3 \times \mathbb{Z}_2$. Let $Z(\bG) = \langle z \rangle$, then $z$ is a generator for the $\mathbb{Z}_2$ component of $A_{\bG}(u)$. We know the action of $F$ on the component group $A_{\btG}(u)$ is trivial so we need only show $F(x) \neq xz$ for all $x \in A_{\bG}(u) - Z_{\bG}(u)$. If $x$ is a 3-cycle of $A_{\bG}(u)$ then we cannot have $F(x) = xz$ because $x$ and $xz$ are of different orders. If $x$ is a 2-cycle such that $F(x) = xz$ then the number of $F$-conjugacy classes of $A_{\bG}(u)$ would be less than 6 and this cannot possibly happen.
\end{proof}

We now deal with the cases of classical type. Let $\bG$ be a spin group, (in other words a simply connected group of type $\B_n$ or $\D_n$), and let $\mathcal{O}$ be a unipotent class of $\bG$. Recall that to every unipotent class of $\bG$ we have a corresponding partition $\lambda$ of $N$, where $N$ is the dimension of the spin group, such that any even number in $\lambda$ occurs an even number of times. This partition is determined by the Jordan blocks of a corresponding class in a special orthogonal group. The non-trivial cases we need to consider occur when the partition $\lambda$ contains at least one odd number and any odd number occurs at most once, (equivalently these are the classes such that $Z(\bG)$ embeds into $A_{\bG}(u)$). We assume now that $\mathcal{O}$ is a class with such a partition. In \cite[\S 14.3]{lusztig:1984:intersection-cohomology-complexes} Lusztig gives a description of the component group $A_{\bG}(u)$, which we recall here.

First of all let $\bar{\bG}$ be a special orthogonal group such that we have an isogeny $\pi : \bG \to \bar{\bG}$. Recall that the kernel of this isogeny is a central subgroup of order 2, which we denote $\Ker(\pi) = \{1,\vartheta\}$. Every non-trivial element of $Z(\bG)$ determines a non-trivial element of $A_{\bG}(u)$, in particular $\vartheta$ determines a non-trivial element of $A_{\bG}(u)$. Let $I := \{a_1,\dots,a_k\}$ be the set of odd numbers occuring in $\lambda$. We let $S$ be the group generated by the elements $\vartheta,x_1,\dots,x_k$, which satisfy the relations
\begin{equation*}
\vartheta^2=1\qquad x_i^2 = \vartheta^{a_i(a_i-1)/2} \qquad x_ix_j = x_jx_i\vartheta\qquad\vartheta x_i = x_i\vartheta
\end{equation*}
for all $i \neq j$. The group $A_{\bG}(u)$ is then isomorphic to the subgroup of $S$ which consists of all elements that can be expressed as a word in an even number of the generators $x_1,\dots,x_k$. For the remainder of this discussion we will identify $A_{\bG}(u)$ with its image in $S$. For all $2 \leqslant i \leqslant k$ let $y_i = x_1x_i$ then $\{\vartheta,y_2,\dots,y_k\}$ is a set of generators for $A_{\bG}(u)$.

Consider the image $\bar{u} = \pi(u)$ of $u$ in the special orthogonal group $\bar{\bG}$. The map $\pi$ induces a surjective map $A_{\bG}(u) \to A_{\bar{\bG}}(\bar{u})$ with kernel $\{1,\vartheta\}$. By \cref{thm:triv-comp-action-adj} and \cref{lem:abelian-comp-group} we know the action of $F$ on $A_{\bar{\bG}}(\bar{u})$ must be trivial for any class representative $u \in \mathcal{O}^F$. Therefore given any element $x \in A_{\bG}(u)$ we have $F(x) \in  \{x,x\vartheta\}$, however we always have $F(\vartheta) = \vartheta$. Any automorphism of $A_{\bG}(u)$ is uniquely determined by its action on the generators, so we break the study of $F$ into two possible cases:

\begin{enumerate}[label=(\alph*)]
	\item $F(y_i) = y_i\vartheta$ for an even number of $y_i$'s,\label{auto:a}
	\item $F(y_i) = y_i\vartheta$ for an odd number of $y_i$'s.\label{auto:b}
\end{enumerate}

In case \ref{auto:a} we claim that the automorphism is an inner automorphism. We start by noticing that if $i_1$, $i_2$, $j_1$, $j_2$ are all distinct then $x_{i_1}x_{i_2}$ and $x_{j_1}x_{j_2}$ commute. Let $\sigma_{ij}$ be the automorphism such that $\sigma_{ij}(y_i) = y_i\vartheta$ and $\sigma_{ij}(y_j) = y_j\vartheta$ but $\sigma_{ij}(y_{\ell}) = y_{\ell}$ whenever $i,j,\ell$ are all distinct. It is easy to verify that $\sigma_{ij}$ is conjugation by $x_ix_j$, so is inner. As $F$ must be a composition of automorphisms of the form $\sigma_{ij}$ we have $F$ is also inner.

We now consider case \ref{auto:b}. Assume $\bG$ is of type $\D_n$ and by composing with a sufficient number of inner automorphisms let us assume that $F$ acts non-trivially on precisely one generator $y_i$. As $F$ acts non-trivially only on $y_i$ it will be true that $A_{\bG}(u)^F$ is the subgroup $\langle \vartheta, y_2,\dots,y_{i-1},y_{i+1},\dots,y_k\rangle$, in particular $|A_{\bG}(u)^F| = 2^{|I|-1}$. As $\bG$ is of type $\D_n$ we have
\begin{equation}\label{eq:auto-odd-gens-swap}
|A_{\bG}(u)^F|/|A_{\btG}(u)| = 2^{|I|-1}/2^{|I|-2} = 2.
\end{equation}
We claim that we also have $|Z_{\bG}(u)^F| = 2$. By the third paragraph of \cite[\S 3.7(a)]{lusztig:2008:irreducible-representations-of-finite-spin-groups} we know the centre of $A_{\bG}(u)$ is given by $\{1,\vartheta,z,z\vartheta\}$ where $z$ is the element $x_1x_2\cdots x_{k-1}x_k$. Note that this means the centre of $A_{\bG}(u)$ coincides with the image of $Z(\bG)$. We have the following expression of $z$ in terms of our chosen generating set
\begin{equation*}
z = x_1x_2\cdots x_k = y_2y_3\cdots y_k\vartheta^c
\end{equation*}
for some $c \in \{0,1\}$. Therefore we have $F(z) = z\vartheta$ and $F(z\vartheta) = z$ so $Z_{\bG}(u)^F = \{1,\vartheta\}$ as required.

Assume now $\bG$ is of type $\B_n$. Recall that $\lambda$ partitions an odd number so $k-1$ is even. By composing with a sufficient number of inner automorphisms let us assume that $F$ acts on all but one generator non-trivially. Assume $y_i$ is the generator such that $F(y_i) = y_i$ then it is easy to check that $y_j^{x_1x_i} = y_j\theta$ whenever $j \neq i$ and clearly $y_i^{x_1x_i} = y_i$. In particular $F$ is an inner automorphism. With this in hand we can prove the following proposition, which is vital for extending \cref{thm:lusztig-hezard}.

\begin{prop}\label{prop:order-comp-grp}
Assume $\bG$ is a simple algebraic group then any $F$-stable unipotent class $\mathcal{O}$ of $\bG$ contains a well-chosen representative.
\end{prop}

\begin{proof}
If $\bG$ is adjoint then the existence of well-chosen representatives follows from \cref{thm:triv-comp-action-adj}, hence we may assume $Z(\bG)$ is disconnected. We have already covered the two trivial cases where $|A_{\bG}(u)| = |A_{\btG}(u)|$ or $|A_{\btG}(u)| = 1$, which covers the case of type $\A_n$. If $\bG$ is an exceptional group then the only cases left to consider are those covered by Lemmas \cref{lem:comps-E6} and \cref{lem:comps-E7}. Assume $\bG$ is a symplectic group or a special orthogonal group of type $\D_n$ then the statement was proved as part of the proof of \cref{thm:triv-comp-action-adj}.

If $\bG$ is a spin group then the result follows from the above discussion and \cref{lem:inner-auto-comp-grp}. Finally assume $\bG$ is a half spin group and $\pi : \bG_{\simc} \to \bG$ is a simply connected cover of $\bG$. We need only comment that case \ref{auto:b} above cannot happen because the Frobenius cannot act non-trivially on the centre of $\bG_{\simc}$ as it must preserve $\Ker(\pi)$. In other words the map $A_{\bG}(u)^F \to A_{\btG}(u)$ is always surjective.
\end{proof}

\section{The Main Result}\label{sec:main-result}
Let $\overline{\mathbb{Q}}_{\ell}$ be an algebraic closure of the field of $\ell$-adic numbers where $\ell>0$ is a prime distinct from $p$. We assume fixed once and for all an automorphism $\overline{\mathbb{Q}}_{\ell} \to \overline{\mathbb{Q}}_{\ell}$ of order two denoted by $x \mapsto \overline{x}$ such that $\overline{\omega} = \omega^{-1}$ for all roots of unity $\omega \in \overline{\mathbb{Q}}_{\ell}^{\times}$. If $H$ is a finite group we denote by $\Cent(H)$ the space of all class functions $f: H \to \overline{\mathbb{Q}}_{\ell}$. This is an inner product space with respect to the standard inner product $\langle - , - \rangle_H : \Cent(H) \times \Cent(H) \to \overline{\mathbb{Q}}_{\ell}$ given by
\begin{equation*}
\langle f,f' \rangle_H = \frac{1}{|H|} \sum_{h \in H}f(h)\overline{f'(h)},
\end{equation*}
(we will write $\langle -,-\rangle$ when $H$ is clear from the context). We have a set of irreducible characters $\Irr(H) \subset \Cent(H)$ which forms an orthonormal basis of $\Cent(H)$. If $K$ is a subgroup of $H$ then we denote by $\Res^H_K : \Cent(H) \to \Cent(K)$ and $\Ind_K^H : \Cent(K) \to \Cent(H)$ the standard restriction and induction maps.

From now on representatives of $F$-stable unipotent classes will be assumed well chosen. Before we can state and prove the main result of this paper we must first recall some facts regarding the character theory of $G$. Let $s \in G^{\star}$ be a semisimple element then we denote by $\mathcal{E}(G,s)$ the \emph{rational} Lustzig series determined by the $G^{\star}$-conjugacy class of $s$. We denote by $C_{G^{\star}}(s)^{\circ}$ the fixed point group ${C_{\bG^{\star}}(s)^{\circ}}^{F^{\star}}$ then this forms a normal subgroup of the full centraliser $C_{G^{\star}}(s)$. As such we have the quotient group $A_{G^{\star}}(s) := C_{G^{\star}}(s)/C_{G^{\star}}(s)^{\circ}$ acts on $\Irr(C_{G^{\star}}(s)^{\circ})$ by conjugation. This action stabilises each Lusztig series and we let $\mathcal{E}(C_{G^{\star}}(s)^{\circ},1)/A_{G^{\star}}(s)$ denote the orbits of the unipotent characters under the action of $A_{G^{\star}}(s)$.

We remark also that as $G$ is a normal subgroup of $\tilde{G}$ we have a natural action of $\tilde{G}/G\cdot Z(\tilde{G})$ on $\Irr(G)$ by conjugation. This action stabilises each Lusztig series of $G$. In \cite[Proposition 5.1]{lusztig:1988:reductive-groups-with-a-disconnected-centre} Lusztig shows that there exists a bijection
\begin{equation}\label{jordan-decomp}
\Psi_s : \mathcal{E}(G,s) \to \mathcal{E}(C_{G^{\star}}(s)^{\circ},1)/A_{G^{\star}}(s)
\end{equation}
such that:
\begin{enumerate}
\renewcommand{\theenumi}{(\roman{enumi})}
	\item the fibres of $\Psi_s$ are the orbits of the action of $\tilde{G}/G\cdot Z(\tilde{G})$ on $\mathcal{E}(G,s)$;
	\item if $\Theta \in \mathcal{E}(C_{G^{\star}}(s)^{\circ},1)/A_{G^{\star}}(s)$ is an orbit and $\Gamma$ is the stabiliser of an element of $\Theta$ in $A_{G^{\star}}(s)$ then $|\Psi_s^{-1}(\Theta)| = |\Gamma|$.
\end{enumerate}
This is a generalisation of the Jordan decomposition of characters. To prove the above result Lusztig shows that the restriction of every irreducible character of $\tilde{G}$ to $G$ is multiplicity free. With this framework in mind we can now state and prove the main result of this article.

\begin{thm}\label{thm:A}
Assume $\bG$ is a simple algebraic group and $\mathcal{O}$ is an $F$-stable unipotent class of $\bG$. There exists a character $\chi \in \Irr(G)$ such that $\Phi_G(\chi) = \mathcal{O}$ and $n_{\chi} = |A_{\bG}(u)^F| = |Z_{\bG}(u)^F||A_{\btG}(u)|$, where $u \in \mathcal{O}^F$ is a well-chosen representative.
\end{thm}

\begin{proof}
If $\bG$ is an adjoint group then this is merely the statement of \cref{thm:lusztig-hezard}, so we may assume $Z(\bG)$ is disconnected. We write $(\tilde{s},\psi)$ for a pair such that $\tilde{s} \in \tilde{G}^{\star}$ is a semisimple element, ($s:=\iota^{\star}(\tilde{s})\in G^{\star}$), and $\psi \in \mathcal{E}({C_{G^{\star}}(s)^{\circ}},1)$ is a unipotent character. We denote by $\tilde{\psi} \in \mathcal{E}(\tilde{G},\tilde{s})$ the character uniquely determined by the sequence of bijections
\begin{equation}\label{eq:bijection}
\mathcal{E}(\tilde{G},\tilde{s}) \to \mathcal{E}(C_{\tilde{G}^{\star}}(\tilde{s}),1) \to \mathcal{E}({C_{G^{\star}}(s)^{\circ}},1),
\end{equation}
where the first bijection comes from the usual Jordan decomposition of characters and the last bijection comes from \cite[Proposition 13.20]{digne-michel:1991:representations-of-finite-groups-of-lie-type}.

Let us assume that $(\tilde{s},\psi)$ is a pair which satisfies the following properties:
\begin{enumerate}[label=(P\arabic*)]
	\item $n_{\psi} = |A_{\tilde{\bG}}(u)|$.\label{P1}
	\item $|\Stab_{A_{G^{\star}}(s)}(\psi)| = |Z_{\bG}(u)^F|$.\label{P2}
	\item $\Phi_{\tilde{G}}(\tilde{\psi}) = \mathcal{O}$, where here we identify $\mathcal{O}$ with its image $\iota(\mathcal{O})$ in $\btG$.\label{P3}
\end{enumerate}
By the Jordan decomposition of characters the character degree of $\psi$ satisfies $\tilde{\psi}(1) = \psi_{\simc}(1)\psi(1)$, where $\psi_{\simc} \in \mathcal{E}(\tilde{G},\tilde{s})$ is the unique semisimple character contained in the Lusztig series. By \cite[Theorem 8.4.8]{carter:1993:finite-groups-of-lie-type} we have $|\tilde{G}|_{p'} = |C_{\tilde{G}}(\tilde{s})|_{p'}\psi_{\simc}(1)$ and by the order formula for finite reductive groups, (see \cite[pg.\ 75]{carter:1993:finite-groups-of-lie-type}), both $|\tilde{G}|_{p'}$ and $|C_{\tilde{G}}(\tilde{s})|_{p'}$ are monic polynomials in $q$. In particular $\psi_{\simc}(1)$ must also be a monic polynomial in $q$ hence $n_{\tilde{\psi}}=n_{\psi} = |A_{\btG}(u)|$ by \ref{P1}. Using the multiplicity free property mentioned above we have $\Res^{\tilde{G}}_G(\tilde{\psi}) = \chi_1 + \cdots + \chi_k$, where $\chi_i \in \mathcal{E}(G,s)$ and $\chi_i(1) = |\Stab_{A_{G^{\star}}(s)}(\psi)|^{-1}\tilde{\psi}(1)$. For any $1 \leqslant i \leqslant k$ we have by \ref{P2} that
\begin{equation*}
n_{\chi_i} = |\Stab_{A_{G^{\star}}(s)}(\psi)|\cdot n_{\tilde{\psi}} = |Z_{\bG}(u)^F||A_{\tilde{\bG}}(u)| = |A_{\bG}(u)|.
\end{equation*}
By \ref{P3} and \cite[Theorem 3.7]{geck-malle:2000:existence-of-a-unipotent-support} we know $\mathcal{O}$ is the unipotent support of each $\chi_i$ hence any $\chi_i$ will satisfy the statement of the theorem.
\end{proof}

Our proof of \cref{thm:A} will be complete once we have verified the following proposition.

\begin{prop}\label{prop:A}
For every simple algebraic group $\bG$ with a disconnected centre and every $F$-stable unipotent class $\mathcal{O}$ of $\bG$ there exists a pair $(\tilde{s},\psi)$, (as specified in the proof of \cref{thm:A}), satisfying \ref{P1} to \ref{P3}. Furthemore $\tilde{s}$ can be chosen such that the image of $s$ under an adjoint quotient of $\bG^{\star}$ is a quasi-isolated semisimple element. Let $\mathcal{F} \subseteq \mathcal{E}(C_{G^{\star}}(s)^{\circ},1)$ be the family of unipotent characters containing $\psi$ then the following condition holds, unless $\bG$ is a spin / half spin group and $A_{\bG}(u)$ is non-abelian:
\begin{equation}\label{eq:clubsuit}
\{\chi \in \mathcal{F} \mid |\Stab_{A_{G^{\star}}(s)^{F^{\star}}}(\chi)| \neq |Z_{\bG}(u)^F|\} = \emptyset.\tag{$\clubsuit$}
\end{equation}
\end{prop}

\noindent We will see that \eqref{eq:clubsuit} will be used in applications below. The detailed case by case verification of \cref{prop:A} is the main result of \cite{taylor:2012:finding-characters-satisfying}, which also forms the main content of the authors PhD thesis. The techniques used are those described in \cite[\S 2]{geck-hezard:2008:unipotent-support} together with Clifford theory.

\begin{rem}
In \cite{taylor:2012:finding-characters-satisfying} the author works with geometric conjugacy classes of semisimple elements with representatives in a fixed maximal torus of $\btG^{\star}$. One can do this if one modifies the action of the Frobenius endomorphism on the centraliser of the semisimple appropriately, (i.e.\ by twisting with an element of the Weyl group -- see \cite[\S2]{taylor:2012:finding-characters-satisfying}). It is an easy exercise for the reader that \cite[Theorem 2.5]{taylor:2012:finding-characters-satisfying} gives a pair as in \cref{prop:A}, (see also the proof of \cite[Theorem A]{taylor:2012:thesis}).
\end{rem}

\section{Generalised Gelfand--Graev Representations}\label{sec:GGGR}
Following \cite{geck-hezard:2008:unipotent-support} we will use the result of \cref{thm:A} to prove Kawanaka's conjecture holds for most simple groups $\bG$ with a disconnected centre. Firstly let us recall that in \cite[\S 3.1]{kawanaka:1986:GGGRs-exceptional} Kawanaka associates to every unipotent element $u \in G$ a GGGR which we denote $\Gamma_u$. This is a representation of $G$ whose construction depends only on the $G$-conjugacy class of $u$. Furthermore, these representations are such that $\Gamma_1$ is the regular representation, (where $1 \in G$ is the identity), and $\Gamma_u$ is a Gelfand--Graev representation if $u$ is a regular unipotent element.

If $\Gamma_u$ is a GGGR of $G$ then we will write $\gamma_u$ for the character of $G$ admitted by $\Gamma_u$ and we use the term GGGR to refer to both $\Gamma_u$ and $\gamma_u$. We can now state Kawanaka's conjecture.

\begin{conj}[Kawanaka, {\cite[3.3.1]{kawanaka:1985:GGGRs-and-ennola-duality}}]\label{conj:kawanaka}
Let $\bG$ be a connected reductive algebraic group and $p$ a good prime for $\bG$. Let $u_1,\dots,u_r$ be representatives for the unipotent conjugacy classes of $G$, then the set $\{\gamma_{u_i} \mid 1 \leqslant i \leqslant r\}$ forms a $\mathbb{Z}$-basis for the $\mathbb{Z}$-module of all unipotently supported virtual characters of $G$.
\end{conj}

\noindent If $\bG$ is connected reductive with a connected centre and simple quotient $\bG/Z(\bG)$ then, under the assumption that $p$, $q$ are large enough, this was proved by Geck and H\'{e}zard in \cite[Theorem 4.5]{geck-hezard:2008:unipotent-support}. Here $p$, $q$ large enough means that the results of \cite{lusztig:1992:a-unipotent-support} are true. We will now assume that $p$, $q$ are large enough so that we may use the results of \cite{geck-hezard:2008:unipotent-support}.

For the following discussion of GGGRs to make sense we must first make some choices. We fix an $F$-stable Borel subgroup $\btB \leqslant \btG$ which contains an $F$-stable maximal torus $\btT \leqslant \btG$. We define $\bB := \btB \cap \bG$ and $\bT := \btT \cap \bG$, then these are similarly such groups for $\bG$. Let $\bU$ be the common unipotent radical of both $\bB$ and $\btB$ then $\btB = \btT \ltimes \bU$ and $\bB = \bT \ltimes \bU$. We denote the fixed points of $\btB$, (resp.\ $\btT$, $\bB$, $\bT$, $\bU$), under the Frobenius endomorphism by $\tilde{B}$, (resp.\ $\tilde{T}$, $B$, $T$, $U$). Our GGGRs are then defined with respect to these choices.

\begin{rem}
By \cite[Remark 2.2]{geck:2004:on-the-schur-indices} we know class representatives for all unipotent classes of $G$ may be found in $U$ hence we assume this to be the case from this point forward. This means we may also assume that our well-chosen representatives lie in $U$, as any $G$-conjugate of a well-chosen representative is well chosen.
\end{rem}

When $Z(\bG)$ is disconnected we will want to relate the GGGRs of $G$ to the GGGRs of $\tilde{G}$. If $u \in \tilde{G}$ is a unipotent element then $u$ is also a unipotent element of $G$. We will denote the GGGR of $\tilde{G}$ associated to $u$ by $\tilde{\gamma}_u$ and the GGGR of $G$ associated to $u$ by $\gamma_u$. With this in mind we have the following observation, which follows immediately from the construction of GGGRs.

\begin{lem}\label{lem:induction-of-GGGRs}
Let $\mathcal{\tilde{O}}$ be a $\tilde{G}$-conjugacy class of unipotent elements and $\{u_i\} \subseteq \mathcal{\tilde{O}}$ a collection of class representatives for the $G$-conjugacy classes contained in $\mathcal{\tilde{O}} \cap G$, then $\tilde{\gamma}_{u_i} = \Ind_{G}^{\tilde{G}}(\gamma_{u_i})$ for each $i$.
\end{lem}

We will need another observation regarding the natural conjugation action of $\tilde{G}$ on a GGGR. The group $\tilde{G}$ is the product $\tilde{T}\cdot G$, therefore we may assume that any left transversal of $G$ in $\tilde{G}$ is contained in $\tilde{T}$. In particular, if $\chi$ is a class function of $G$ then for any $g \in \tilde{G}$ there exists $t \in \tilde{T}$ such that $\chi^g = \chi^t$, (where $\chi^g(x) = \chi(gxg^{-1})$ for all $x \in G$). With this notation we have the following result.

\begin{prop}[Geck, {\cite[Proposition 2.2]{geck:1993:basic-sets-II}}]\label{prop:gggr-conj}
For any unipotent element $u \in U$ and any $t \in \tilde{T}$ we have $\gamma_u^t = \gamma_{tut^{-1}}$.
\end{prop}

\section{The Case of Abelian Component Groups}\label{sec:ab-comp-groups}
To prove Kawanaka's conjecture we will follow the same line of argument as in \cite[\S 4]{geck-hezard:2008:unipotent-support}. In particular the focus will be on the following observation.

\begin{lem}[Geck--H\'{e}zard, {\cite[Lemma 4.2]{geck-hezard:2008:unipotent-support}}]\label{lem:matrix-mult-invert}
Let $u_1,\dots,u_d$ be representatives for the unipotent conjugacy classes of $G$. Assume that there exist virtual characters $\chi_1,\dots,\chi_d$ of $G$ such that the matrix of scalar products $(\langle \chi_i,\gamma_{u_j} \rangle_{G})_{1 \leqslant i,j \leqslant d}$ is invertible over $\mathbb{Z}$, then \cref{conj:kawanaka} holds.
\end{lem}

\noindent We will start by proving a result which is crucial in dealing with almost all cases. Let $D_G : \Cent(G) \to \Cent(G)$ denote the Alvis--Curtis duality map. It is known that $D_G$ is an isometry of $\Cent(G)$ and $D_G \circ D_G$ is the identity, (see for example \cite[\S 8]{digne-michel:1991:representations-of-finite-groups-of-lie-type}). Hence, for any irreducible character $\chi \in \Irr(G)$ there exists a sign $\epsilon_{\chi} \in \{1,-1\}$ such that $\epsilon_{\chi}D_G(\chi) \in \Irr(G)$. For any $\chi$ we denote $\epsilon_{\chi}D_G(\chi)$ by $\chi^*$. It will be useful for us to know the following relationship between the Alvis--Curtis duality maps $D_{\tilde{G}}$ and $D_G$.

\begin{lem}\label{lem:duality-restriction}
For any connected reductive algebraic group $\bG$ we have
\begin{equation*}
\Res_{G}^{\tilde{G}} \circ D_{\tilde{G}} = D_G \circ \Res_{G}^{\tilde{G}}.
\end{equation*}
\end{lem}

\begin{proof}
If $\bH$ is a connected reductive algebraic group defined over $\mathbb{F}_q$ then we denote by $\rk(\bH)$ its semisimple $\mathbb{F}_q$-rank, (see \cite[Definition 8.6]{digne-michel:1991:representations-of-finite-groups-of-lie-type}). Also if $\bL \leqslant \bG$ is an $F$-stable Levi subgroup contained in an $F$-stable subgroup of $\bG$ then we denote by $R_{\bL}^{\bG}$ the corresponding Harish-Chandra induction map. By \cite[Proposition 2.2]{digne-michel:1991:representations-of-finite-groups-of-lie-type} if $\tilde{\bH}$ is a Borel, parabolic or Levi subgroup of $\btG$ then $\bH:=\tilde{\bH}\cap\bG$ is also such a subgroup of $\bG$ and the map $\tilde{\bH} \to \bH$ gives a bijection between the sets of such subgroups. Identifying these groups in this way and using the statements in \cite[Proposition 10.10]{bonnafe:2006:sln} we have
\begin{align*}
\Res_G^{\tilde{G}} \circ D_{\btG} &= \sum_{\btP \geqslant \btB} (-1)^{\rk(\btP)} (\Res_G^{\tilde{G}} \circ R_{\btL}^{\btG} )\circ {}^*R_{\btL}^{\btG},\\
&= \sum_{\btP \geqslant \btB} (-1)^{\rk(\btP)} R_{\bL}^{\bG} \circ (\Res_L^{\tilde{L}}\circ {}^*R_{\btL}^{\btG}),\\
&= \sum_{\bP \geqslant \bB} (-1)^{r(\bP)} (R_{\bL}^{\bG} \circ {}^*R_{\bL}^{\bG}) \circ \Res_G^{\tilde{G}},\\
&= D_{\bG} \circ \Res_G^{\tilde{G}}.
\end{align*}
To obtain the third equality we have used the fact that the inclusion morphism $\bP \to \btP$ induces an isomorphism between the derived subgroups of $\bP$ and $\btP$ which is defined over $\mathbb{F}_q$, hence $\rk(\bP) = \rk(\btP)$.
\end{proof}

\begin{cor}\label{cor:dual-restrict}
Let $\tilde{\chi} \in \Irr(\tilde{G})$ and $\chi_i \in \Irr(G)$ be such that
\begin{equation*}
\Res^{\tilde{G}}_{G}(\tilde{\chi}) = \chi_1+\cdots+\chi_k,
\end{equation*}
then $\epsilon_{\tilde{\chi}} = \epsilon_{\chi_i}$ for all $1 \leqslant i \leqslant k$. In particular $\Res^{\tilde{G}}_G(\tilde{\chi}^*) = \chi_1^* + \cdots + \chi_k^*$.
\end{cor}

\begin{proof}
We know $\epsilon_{\tilde{\chi}}D_{\tilde{G}}(\tilde{\chi}) \in \Irr(\tilde{G})$, in particular it is a character of $\tilde{G}$. This means
\begin{equation*}
\Res_{G}^{\tilde{G}}(\tilde{\chi}^*) = \epsilon_{\tilde{\chi}}(\Res_{G}^{\tilde{G}}\circ D_{\tilde{G}})(\tilde{\chi}) = \epsilon_{\tilde{\chi}}(D_G \circ\Res_{G}^{\tilde{G}})(\tilde{\chi}) = \epsilon_{\tilde{\chi}}D_G(\chi_1) + \cdots + \epsilon_{\tilde{\chi}}D_G(\chi_k).
\end{equation*}
is a character of $G$, where the second equality is obtained by \cref{lem:duality-restriction}. As it is a character all coefficients of irreducible constituents must be positive, therefore we cannot have $\epsilon_{\chi_i} \neq \epsilon_{\tilde{\chi}}$ for any $1 \leqslant i \leqslant k$. The final statement is clear by definition.
\end{proof}

We temporarily fix the following notation. If $\mathcal{O}$ is an $F$-stable unipotent class of $\bG$ then we denote by $\mathcal{\tilde{O}}_i$, for $1 \leqslant i \leqslant d$, the $\tilde{G}$-classes such that $\mathcal{O}^F = \mathcal{\tilde{O}}_1 \sqcup \dots \sqcup \mathcal{\tilde{O}}_d$. For each $\mathcal{\tilde{O}}_i$ we write $\mathcal{O}_{i,j}$, for $1 \leqslant j \leqslant k_i$ (where $k_i$ is a number depending upon the class $\mathcal{O}_i$), for the $G$-classes such that $\mathcal{\tilde{O}}_i \cap G = \mathcal{O}_{i,1} \sqcup \cdots \sqcup \mathcal{O}_{i,k_i}$. Finally we fix $G$-class representatives $u_{i,j} \in \mathcal{O}_{i,j}$ for each $1 \leqslant i \leqslant d$ and $1 \leqslant j \leqslant k_i$.

\begin{prop}\label{prop:inner-prod}
Let $\mathcal{O}$ be an $F$-stable unipotent class of $\bG$ and assume $A_{\bG}(u)$ is abelian. Let $(\tilde{s},\psi)$ be the pair prescribed by \cref{prop:A}, then there exist irreducible characters $\chi_{i,j} \in \mathcal{E}(G,s)$ such that $\langle \chi_{i,j}^*, \gamma_{u_{x,y}} \rangle = \delta_{i,x}\delta_{j,y}$, for $1 \leqslant i,x \leqslant d$ and $1 \leqslant j,y \leqslant k_i$, (here $\delta_{*,*}$ denotes the Kronecker delta). Furthermore $k_i = |Z_{\bG}(u)^F|$ for all $i$.
\end{prop}

\begin{proof}
Recall that $u \in \mathcal{O}^F$ is a class representative such that $A_{\btG}(u)^F = A_{\btG}(u)$. The $\tilde{G}$-classes in $\mathcal{O}^F$ are in bijection with the $F$-conjugacy classes in $A_{\btG}(u)$, therefore $d=|A_{\btG}(u)|$ by \cite[Exemple 1.1]{bonnafe:2006:sln}. Let $\mathcal{\tilde{F}} \subseteq \mathcal{E}(\tilde{G},\tilde{s})$ be the family of characters which, under the map in \eqref{eq:bijection} is in bijection with the family of characters containing $\psi$.

By \cite[Proposition 4.3]{geck-hezard:2008:unipotent-support}, which we can use as \ref{P1} and \ref{P3} hold, there exist irreducible characters $\tilde{\chi}_{1,1},\dots,\tilde{\chi}_{d,1} \in \tilde{\mathcal{F}}$ such that $\langle \tilde{\chi}_{i,1}^*,\tilde{\gamma}_{u_{x,1}}\rangle = \delta_{i,x}$ for all $1 \leqslant i,x \leqslant d$. Using \cref{lem:induction-of-GGGRs} followed by Frobenius reciprocity we see that
\begin{equation*}
\delta_{i,x} = \langle \tilde{\chi}_{i,1}^*,\tilde{\gamma}_{u_{x,1}}\rangle_{\tilde{G}} = \langle \tilde{\chi}_{i,1}^*,\Ind_G^{\tilde{G}}(\gamma_{u_{x,1}})\rangle_{\tilde{G}} = \langle \Res_G^{\tilde{G}}(\tilde{\chi}_{i,1}^*),\gamma_{u_{x,1}}\rangle_G.
\end{equation*}
As $A_{\bG}(u)$ is abelian we know \eqref{eq:clubsuit} holds, in particular using \cref{cor:dual-restrict} we know the restriction has the following decomposition into irreducible characters
\begin{equation*}
\Res_G^{\tilde{G}}(\tilde{\chi}_{i,1}^*) = \chi_{i,1}^* + \cdots + \chi_{i,|Z_{\bG}(u)^F|}^* \Rightarrow \delta_{i,x} = \sum_{j=1}^{|Z_{\bG}(u)^F|} \langle \chi_{i,j}^*,\gamma_{u_{x,1}} \rangle.
\end{equation*}
Without loss of generality we may assume the labelling to be such that $\langle \chi_{i,j}^*, \gamma_{u_{x,1}} \rangle = \delta_{i,x}\delta_{j,1}$. We will write $\Stab_{\tilde{G}}(\chi_{i,1}^*)$ for the stabiliser of $\chi_{i,1}^*$ in $\tilde{G}$ under the natural conjugation action of $\tilde{G}$ on $\Irr(G)$. By Clifford theory and the remark before \cref{prop:gggr-conj} there exists a set $\{t_{i,1},\dots,t_{i,|Z_{\bG}(u)^F|}\} \subseteq \tilde{T}$, (which can be completed to form a left transversal of $\Stab_{\tilde{G}}(\chi_{i,1}^*)$ in $\tilde{G}$), such that $\chi_{i,j}^* := {\chi_{i,1}^*}^{t_{i,j}}$ satisfies the condition $\chi_{i,j}^* = \chi_{i,k}^*$ if and only if $j = k$. We assume for convenience that $t_{i,1}$ is the identity for all $i$. As conjugation by elements of $\tilde{G}$ is an isometry of the space of all class functions
\begin{equation*}
\delta_{i,x} = \langle {\chi_{i,1}^*}^{t_{i,j}}, \gamma_{u_{x,1}}^{t_{i,j}} \rangle = \langle \chi_{i,j}^*, \gamma_{t_{i,j}(u_{x,1})t_{i,j}^{-1}}\rangle.
\end{equation*}

We claim that if $j$, $k$ are distinct indices then we cannot have $t_{i,j}(u_{x,1})t_{i,j}^{-1}$ and $t_{i,k}(u_{x,1})t_{i,k}^{-1}$ are in the same $G$-class. If they were in the same $G$-class then we would have $\gamma_{u_{x,1}}^{t_{i,j}} = \gamma_{u_{x,1}}^{t_{i,k}}$ by \cref{prop:gggr-conj}, which would mean
\begin{equation*}
\delta_{i,x} = \langle {\chi_{i,1}^*}^{t_{i,k}}, \gamma_{u_{x,1}}^{t_{i,j}} \rangle = \langle {\chi_{i,1}^*}^{t_{i,k}t_{i,j}^{-1}}, \gamma_{u_{x,1}} \rangle.
\end{equation*}
However of all the components of $\Res^{\tilde{G}}_G(\tilde{\chi}_{i,1}^*)$ only $\chi_{i,1}^*$ satisfies this property. So $\chi_{i,1}^* = {\chi_{i,1}^*}^{t_{i,k}t_{i,j}^{-1}}$, which would imply $t_{i,k}$ and $t_{i,j}$ have the same image in the quotient $\tilde{G}/\Stab_{\tilde{G}}(\chi_{i,1}^*)$ but this is a contradiction.

This argument shows that there are at least $|Z_{\bG}(u)^F|$ conjugacy classes contained in $\mathcal{\tilde{O}}_i \cap G$ for each $1 \leqslant i \leqslant d$, in other words $|Z_{\bG}(u)^F| \leqslant k_i$. Conversely this clearly gives un an inequality
\begin{equation*}
|A_{\bG}(u)^F| = |Z_{\bG}(u)^F||A_{\tilde{\bG}}(u)| = d|Z_{\bG}(u)^F| \leqslant \sum_{i=1}^{d} k_i = |A_{\bG}(u)^F|
\end{equation*}
so we must have $k_i = |Z_{\bG}(u)^F|$ for all $i$. By this argument we may now redefine our class representatives to be such that, for all $1 \leqslant x \leqslant d$ and $1 \leqslant y \leqslant |Z_{\bG}(u)^F|$, we have $u_{x,y} := t_{x,y}u_{x,1}t_{x,y}^{-1}$. With all this in mind the statement of the proposition is now simple. Indeed we first see that
\begin{equation*}
\langle \chi_{i,j}^*, \gamma_{u_{x,y}} \rangle = \langle {\chi_{i,1}^*}^{t_{i,j}}, \gamma_{u_{x,1}}^{t_{x,y}} \rangle = \langle {\chi_{i,1}^*}^{t_{i,j}t_{x,y}^{-1}}, \gamma_{u_{x,1}} \rangle
\end{equation*}
but it is clear that $\langle {\chi_{i,1}^*}^{t_{i,j}t_{x,y}^{-1}}, \Res_G^{\tilde{G}}(\tilde{\chi}_{i,1}^*) \rangle = 1$, which tells us the above inner product is $0$ unless $i=x$. Now assume $i = x$ then $\langle {\chi_{i,1}^*}^{t_{i,j}t_{i,y}^{-1}}, \gamma_{u_{i,1}} \rangle  = 1$ if and only if ${\chi_{i,1}^*}^{t_{i,j}t_{i,y}^{-1}} = \chi_{i,1}^*$, which is true if and only if $j = y$ so we are done.
\end{proof}

For convenience we restate the above proposition with slightly less cumbersome notation.

\begin{cor}\label{cor:conj-abelian-comp}
Let $\mathcal{O}$ be an $F$-stable unipotent class of $\bG$ and assume $A_{\bG}(u)$ is abelian. Let $d' := |A_{\bG}(u)^F|$ and denote by $\mathcal{O}_i$ the $G$-classes such that $\mathcal{O}^F = \mathcal{O}_1\sqcup \cdots \sqcup \mathcal{O}_{d'}$, furthermore let $u_i \in \mathcal{O}_i$ denote a class representative. Consider the pair $(\tilde{s},\psi)$ prescribed by \cref{prop:A} then there exist irreducible characters $\chi_i \in \mathcal{E}(G,s)$ such that $\langle \chi_i^*, \gamma_{u_j} \rangle = \delta_{i,j}$, for all $1 \leqslant i,j \leqslant d'$.
\end{cor}

\begin{rem}
Let us relax the condition that $A_{\bG}(u)$ is abelian in the above result but instead assume $A_{\tilde{\bG}}(u)$ is abelian. In particular let $\bG$ be a spin / half spin group and $u$ be such that $Z_{\bG}(u) = Z(\bG)$ then we are in the following situation. The family $\tilde{\mathcal{F}}$ described in the proof of \cref{prop:inner-prod} has order $m^2$, where $m$ is a power of 2. The group $A_{\btG}(u)$ is abelian so we may use \cite[Proposition 4.3]{geck-hezard:2008:unipotent-support}, as in the above proof, to obtain $m$ characters in $\tilde{\mathcal{F}}$ satisfying the inner product condition. In this situation \eqref{eq:clubsuit} fails and in fact there are \emph{only} $m$ characters in $\mathcal{\tilde{F}}$ whose restriction to $G$ contains the correct number of irreducible constituents, (see \cite[Propositions 9.3 and 11.10]{taylor:2012:finding-characters-satisfying}). The problem is that it is not clear that the characters provided by the result of Geck and H\'{e}zard coincide with the characters which have the correct restriction from $\tilde{G}$ to $G$. This is the key obstruction to showing Kawanaka's conjecture holds for spin / half spin groups.
\end{rem}

One way to fix this problem would be to make sure that the character sheaves used in the proof of \cite[Corollary 3.5]{geck-hezard:2008:unipotent-support} can be chosen to have the same labelling as the irreducible characters in $\mathcal{\tilde{F}}$ whose restriction to $G$ is not irreducible. The author considered this in the case where $\bG$ is a spin group of type $\B_n$ by trying to adapt the explicit results obtained in \cite{lusztig:1986:on-the-character-values}. However it seems that trying to understand the correspondence between the two labellings is significantly complicated.

\section{Kawanaka's Conjecture}\label{sec:conjecture}
The result in the previous section will be crucial in determining Kawanaka's conjecture when $\bG$ is of classical type. If $\bG$ is of exceptional type and $Z(\bG)$ is disconnected then we must have $\bG$ is of type $\E_6$ or $\E_7$. In these groups there are only three classes such that $A_{\bG}(u)$ is non-abelian and they are all such that $A_{\btG}(u) \cong \mathfrak{S}_3$. Assume $\mathcal{O}$ is one of these three classes and $(\tilde{s},\psi)$ is the pair prescribed by \cref{prop:A}.

Write $\mathcal{\tilde{O}}_1$, $\mathcal{\tilde{O}}_2$ and $\mathcal{\tilde{O}}_3$ for the $\tilde{G}$-classes such that $\mathcal{O}^F = \mathcal{\tilde{O}}_1 \sqcup \mathcal{\tilde{O}}_2 \sqcup \mathcal{\tilde{O}}_3$ and fix representatives $u_i \in \mathcal{\tilde{O}}_i \cap G$. The matrix of multiplicities between the irreducible characters in $\mathcal{E}(\tilde{G},\tilde{s})$ and the Alvis--Curtis duals of the characters of the GGGRs is given in \cite[Proposition 6.7]{geck:1999:character-sheaves-and-GGGRs}. Using this table we see that there exist three characters $\tilde{\chi}_1$, $\tilde{\chi}_2$, $\tilde{\chi}_3 \in \Irr(\tilde{G})$ such that $\langle D_{\tilde{G}}(\tilde{\chi}_i),\tilde{\gamma}_{u_j}\rangle = \langle\tilde{\chi}_i,D_{\tilde{G}}(\tilde{\gamma}_{u_j})\rangle = \delta_{i,j}$. In particular $\epsilon_{\tilde{\chi}} = 1$ for all $\tilde{\chi} \in \mathcal{E}(\tilde{G},\tilde{s})$, which means $\langle\tilde{\chi}_i^*,\tilde{\gamma}_{u_j}\rangle = \delta_{i,j}$. Two out of three of the classes satisfy $A_{\btG}(u) \cong A_{\bG}(u)$. If this is the case then $\mathcal{O}_i := \mathcal{\tilde{O}}_i \cap G$ is a single $G$-class and $\mathcal{O}^F = \mathcal{O}_1\sqcup\mathcal{O}_2\sqcup\mathcal{O}_3$. Let $u_i \in \mathcal{O}_i$ denote a $G$-class representative then as \eqref{eq:clubsuit} holds we know $\chi_i := \Res_G^{\tilde{G}}(\tilde{\chi}_i)$ is irreducible. In particular for each $1 \leqslant i,j \leqslant 3$ we have $\langle\chi_i^*,\gamma_{u_j}\rangle = \delta_{i,j}$.

For the remaining class we have $|A_{\bG}(u)| = 2|A_{\btG}(u)|$. The number of $F$-conjugacy classes in $A_{\bG}(u)$ is equal to $6 = |A_{\bG}(u)^F|$. We know \eqref{eq:clubsuit} holds for this class, so for each $1 \leqslant i \leqslant 3$ we have $\Res_G^{\tilde{G}}(\tilde{\chi}_i) = \chi_{i,1} + \chi_{i,2}$ where $\chi_{i,j} \in \Irr(G)$. Using the techniques in the proof of \cref{prop:inner-prod} it is easy to see that each of the three $\tilde{G}$-classes $\mathcal{\tilde{O}}_i$ is such that $\mathcal{\tilde{O}}_i^F \cap G = \mathcal{O}_{i,1}\sqcup \mathcal{O}_{i,2}$, where $\mathcal{O}_{i,j}$ is a $G$-class for $1 \leqslant j \leqslant 2$. Let us choose class representatives $u_{x,y} \in \mathcal{O}_{x,y}$ for each $1 \leqslant x \leqslant 3$ and $1 \leqslant y \leqslant 2$ then the following relation holds $\langle\chi_{i,j}^*, \gamma_{u_{x,y}}\rangle = \delta_{i,x}\delta_{j,y}$. In particular we have the following corollary.

\begin{cor}\label{cor:conj-S3-comp}
Let $\mathcal{O}$ be an $F$-stable unipotent class of $\bG$ and assume $A_{\btG}(u) \cong \mathfrak{S}_3$. Let us write $\mathcal{O}_i$, for $1 \leqslant i \leqslant d'$ for the $G$-classes such that $\mathcal{O}^F = \mathcal{O}_1\sqcup \cdots \sqcup \mathcal{O}_{d'}$ and let $u_i \in \mathcal{O}_i$ denote a class representative. Consider the pair $(\tilde{s},\psi)$ prescribed by \cref{prop:A} then there exist irreducible characters $\chi_i \in \mathcal{E}(G,s)$ such that $\langle \chi_i^*, \gamma_{u_j} \rangle = \delta_{i,j}$ for all $1 \leqslant i,j \leqslant d'$.
\end{cor}

We have now gathered all the preliminary information that we need to prove the following theorem.

\begin{thm}
Assume $p$, $q$ are large enough and $\bG$ is a simple algebraic group, which is not a spin or half-spin group, then \cref{conj:kawanaka} holds.
\end{thm}

\begin{proof}
If $\bG$ is adjoint then this is just the result obtained by Geck and H\'{e}zard so we may assume that $\bG$ has a disconnected centre. Let $\mathcal{O}_1,\dots,\mathcal{O}_d$ be the distinct $F$-stable unipotent classes of $\bG$ and for each $1\leqslant i \leqslant d$ let $(\tilde{s}_i,\psi_i)$ be the pair prescribed by \cref{prop:A} for $\mathcal{O}_i$. For each $1 \leqslant i \leqslant d$ we write $\mathcal{O}_{i,1},\dots,\mathcal{O}_{i,k_i}$ for the $G$-conjugacy classes such that $\mathcal{O}_i^F = \mathcal{O}_{i,1}\sqcup\dots\sqcup\mathcal{O}_{i,k_i}$. By \cref{cor:conj-abelian-comp} and \cref{cor:conj-S3-comp} we can find irreducible characters $\{\chi_{i,1},\dots,\chi_{i,k_i}\} \subseteq \mathcal{E}(G,s_i)$ such that $\langle \chi_{i,j}^*, \gamma_{u_{i,y}} \rangle = \delta_{j,y}$. In particular the matrix of multiplicities $(\langle\chi_{i,j}^*,\gamma_{u_{x,y}}\rangle)$, (where we have $1\leqslant i,x \leqslant d$, $1\leqslant j \leqslant k_i$, $1\leqslant y \leqslant k_x$), is a square block matrix with identity matrices on the diagonal. We may now argue as in the proof of \cite[Theorem 4.5]{geck-hezard:2008:unipotent-support} to show that all blocks in the lower triangular part of this matrix are zero. In particular this matrix of multiplicities is invertible over $\mathbb{Z}$ so by \cref{lem:matrix-mult-invert} Kawanaka's conjecture holds.
\end{proof}

\begin{rem}
In \cite[Proposition 4.6]{geck-hezard:2008:unipotent-support} Geck and H\'{e}zard give a characterisation of GGGRs in terms of character values. This characterisation follows as a formal consequence of Kawanaka's conjecture, so this now also holds for all simple groups which are not a spin / half spin group.
\end{rem}

\section{GGGRs and Fourth Roots of Unity}\label{sec:fourth-roots}
In this section we would like to use \cref{thm:A} together with the techniques of \cite[\S 3]{geck:1999:character-sheaves-and-GGGRs} to compute certain fourth roots of unity that arise in connection to GGGRs. We begin by recalling the setup of \cite{geck:1999:character-sheaves-and-GGGRs}. Let $\mathcal{N}_{\bG}$ denote the set of all pairs $\iota = (\mathcal{O}_{\iota},\psi_{\iota})$, where $\mathcal{O}_{\iota}$ is a unipotent conjugacy class of $\bG$ and $\psi_{\iota} \in \Irr(A_{\bG}(u))$ for $u \in \mathcal{O}_{\iota}$. We say a pair $\iota$ is $F$-stable if $F(\mathcal{O}_{\iota}) = \mathcal{O}_{\iota}$ and $\psi_{\iota}\circ F = \psi_{\iota}$. Here we implicitly assume $u \in \mathcal{O}^F$ so that $F$ induces an automorphism of $A_{\bG}(u)$. We denote the subset of $F$-stable pairs by $\mathcal{N}_{\bG}^F \subseteq \mathcal{N}_{\bG}$.

Assume now that $\iota$ is an $F$-stable pair. As $\psi_{\iota}$ is invariant under the action of $F$ we may extend $\psi_{\iota}$ to a character of the semidirect product $A_{\bG}(u) \rtimes \langle F \rangle$, where $F$ acts as a cyclic group of automorphisms. Let $\tilde{\psi}_{\iota}$ be a fixed choice of such an extension. For each $x \in A_{\bG}(u)$ we write $u_x$ for an element of $\mathcal{O}_{\iota}^F$ obtained by twisting $u$ with the element $x \in A_{\bG}(u)$. We then define
\begin{equation*}
Y_{\iota}(g) = \begin{cases}
\tilde{\psi}_{\iota}(xF) &\text{if }g = u_x\text{ for some }x \in A_{\bG}(u),\\
0 &\text{otherwise}
\end{cases}
\end{equation*}
for all $g \in G$. The function $Y_{\iota}$ is a $G$-class function and the set $\mathcal{Y} = \{Y_{\iota} \mid \iota \in \mathcal{N}_{\bG}^F\}$ forms a basis for the space of all unipotently supported class functions of $G$.

Let $\iota \in \mathcal{N}_{\bG}^F$ be an $F$-stable pair and let $u_1,\dots,u_d \in \mathcal{O}_{\iota}^F$ be class representatives for the $G$-classes contained in $\mathcal{O}_{\iota}^F$. We define, as in \cite[7.5]{lusztig:1992:a-unipotent-support}, for any pair $\iota \in \mathcal{N}_{\bG}^F$ the $G$-class function
\begin{equation*}
\gamma_{\iota} = \sum_{r=1}^d [ A_{\bG}(u_r) : A_{\bG}(u_r)^F]Y_{\iota}(u_r)\gamma_{u_r}.
\end{equation*}
The set $\mathcal{X} = \{\gamma_{\iota} \mid \iota \in \mathcal{N}_{\bG}^F\}$ is then also a basis for the space of all unipotently supported class functions of $G$.

Recall that in \cite{lusztig:1984:intersection-cohomology-complexes} Lusztig has associated to every pair $\iota \in \mathcal{N}_{\bG}$ a unique pair $(\bL_{\iota},\upsilon_{\iota})$, up to $\bG$ conjugacy, where $\bL_{\iota}$ is a Levi subgroup of $\bG$ and $\upsilon_{\iota} \in \mathcal{N}_{\bL_{\iota}}$ is a cuspidal pair. This is known as the generalised Springer correspondence. Lusztig has shown that we have a disjoint union
\begin{equation*}
\mathcal{N}_{\bG} = \bigsqcup_{(\bL,\upsilon)} \bl{I}(\bL,\upsilon) \qquad\text{where}\qquad \bl{I}(\bL,\upsilon) = \{\iota \in \mathcal{N}_{\bG} \mid (\bL_{\iota},\upsilon_{\iota}) = (\bL,\upsilon)\}.
\end{equation*}
We call $\bl{I}(\bL,\upsilon)$ a \emph{block} of $\mathcal{N}_{\bG}$. A pair $\iota \in \mathcal{N}_{\bG}$ is a cuspidal pair if and only if it lies in a block of cardinality 1. To each pair $\iota \in \mathcal{N}_{\bG}$ we also assign the following value
\begin{equation*}
b_{\iota} = \frac{1}{2}(\dim \bG - \dim \mathcal{O}_{\iota} - \dim Z(\bL_{\iota})^{\circ}).
\end{equation*}

In \cite[Theorem 7.3]{lusztig:1992:a-unipotent-support} Lusztig explicitly constructs the change of basis from $\mathcal{X}$ to $\mathcal{Y}$. The expression of the function $\gamma_{\iota} \in \mathcal{X}$ in terms of elements of $\mathcal{Y}$ involves an unknown fourth root of unity $\zeta_{\iota}$. This root of unity is defined in \cite[Proposition 7.2]{lusztig:1992:a-unipotent-support} and it is shown there that $\zeta_{\iota} = \zeta_{\iota'}$ whenever $\iota,\iota'$ belong to the same block. Following \cite[8.4]{lusztig:1992:a-unipotent-support} to each $\iota$ we define $\delta_{\iota} = (-1)^{\rank(\bL_{\iota}/Z(\bL_{\iota}))}$ and $\zeta_{\iota}' = \delta_{\iota}\zeta_{\iota}^{-1}$. As $\delta_{\iota}$ and $\zeta_{\iota}$ only depend on the block containing $\iota$ the same must be true for $\zeta_{\iota}'$. Moreover we have the following more precise statement.

\begin{lem}[Lusztig, {\cite[Proposition 7.2]{lusztig:1992:a-unipotent-support}}]
Let $\bl{I}(\bL,\upsilon)$ be a block of $\mathcal{N}_{\bG}$ and assume that $\upsilon \in \mathcal{N}_{\bL}^F$. If $\zeta_{\upsilon}'$ is the fourth root of unity associated to $\upsilon$ in $\bL^F$ then $\zeta_{\iota}' = \zeta_{\upsilon}'$ for all $\iota \in \bl{I}(\bL,\upsilon)$.
\end{lem}

There is one extreme case where we always know the value $\zeta_{\iota}'$. If $\bL_{\iota}$ is a maximal torus then from the definitions it is clear that $\zeta_{\iota}' = 1$. The results of \cite{geck:1999:character-sheaves-and-GGGRs} are based on the following two axioms, (these are known to be true when $p$ and $q$ are sufficiently large).
\begin{enumerate}[label=(A\arabic*)]
	\item Let $\iota,\iota' \in \mathcal{N}_{\bG}^F$ be such that $\mathcal{O}_{\iota} = \mathcal{O}_{\iota'}$ then $\langle D_G(\gamma_{\iota}), Y_{\iota'}\rangle = |A_{\bG}(u)|\zeta_{\iota}'q^{-b_{\iota}}\delta_{\iota,\iota'}$, where $\delta_{\iota,\iota'}$ denotes the Kronecker delta.\label{ax:1}
	\item If $D_G(\gamma_{\iota})(x) \neq 0$ for some $x \in G$ then $\langle x\rangle \leqslant \mathcal{O}_{\iota}$ where $\langle x \rangle$ is the $G$-conjugacy class containing $x$.\label{ax:2}
\end{enumerate}
As we will use the work of Geck we will also take these as axioms upon which our arguments are built. Our main focus in this section is to show that using \cref{thm:A} statements similar to \cite[Theorem 3.8]{geck:1999:character-sheaves-and-GGGRs} can be made for symplectic and special orthogonal groups. To do this we need the following easy extensions of results of Geck.
\begin{prop}\label{prop:GGGR-mult}
Let $\mathcal{O}$ be an $F$-stable unipotent class of $\bG$ and consider the pair $(\tilde{s},\psi)$ prescribed by \cref{prop:A}. Let $d'$, (resp.\ $d$), be the number of $\tilde{G}$, (resp.\ $G$), conjugacy classes contained in $\tilde{\mathcal{O}}^F$. There exists a character $\chi \in \mathcal{E}(G,s)$ and an index $1 \leqslant x \leqslant d$ such that
\begin{equation*}
\langle \chi^*, \gamma_{u_y} \rangle = \delta_{x,y}\qquad\text{and}\qquad \langle \chi^*, \gamma_{\iota} \rangle = \frac{|Z_{\bG}(u)|}{|Z_{\bG}(u)^F|}\cdot\overline{Y_{\iota}(u_x)},
\end{equation*}
for all $\iota \in \mathcal{N}_{\bG}^F$ with $\mathcal{O}_\iota = \mathcal{O}$.
\end{prop}
\begin{proof}
The proof of this result is similar in nature to that of \cref{prop:inner-prod}. Let us adopt the notational conventions introduced immediately before \cref{prop:inner-prod}. Let $\tilde{\psi} \in \mathcal{E}(\tilde{G},\tilde{s})$ be the character corresponding to $\psi$ under the bijection in \eqref{eq:bijection}, then by \cite[Proposition 3.1]{geck:1999:character-sheaves-and-GGGRs} there exists an index $1 \leqslant x \leqslant d'$ such that $\langle \tilde{\psi}^*,\tilde{\gamma}_{u_y} \rangle_{\tilde{G}} = \delta_{x,y}$ for all $1 \leqslant y \leqslant d'$.

Using the same argument as in \cref{prop:inner-prod} we see there is a unique irreducible constituent $\chi$ in $\Res_G^{\tilde{G}}(\tilde{\psi})$ satisfying $\langle \chi^*, \gamma_{u_y} \rangle_G = \delta_{x,y}$ for all $1 \leqslant y \leqslant d$. By \cite[2.8(b)]{geck:1999:character-sheaves-and-GGGRs} we have
\begin{equation*}
\langle \chi^*,\gamma_{(\mathcal{O},1)} \rangle = \frac{|A_{\bG}(u)|}{n_{\chi}} \Rightarrow \langle \chi^*,\gamma_{(\mathcal{O},1)} \rangle = \frac{|A_{\bG}(u)|}{|A_{\bG}(u)^F|} = \frac{|Z_{\bG}(u)|}{|Z_{\bG}(u)^F|}.
\end{equation*}
From the definition of $\gamma_{\iota}$ we conclude that
\begin{equation*}
\langle \chi^*,\gamma_{\iota} \rangle = \left\langle \chi^*,\sum_{r=1}^d \frac{|A_{\bG}(u)|}{|A_{\bG}(u_r)^F|}Y_i(u_r)\gamma_{u_r} \right\rangle = \frac{|Z_{\bG}(u)|}{|Z_{\bG}(u)^F|}\cdot\overline{Y_i(u_x)}
\end{equation*}
because $\chi$ only has non-zero multiplicity in $\gamma_{u_x}$.
\end{proof}

\begin{cor}\label{cor:4th-root-diviser}
Let $\mathcal{O}$, $\chi$ and $x$ be as in \cref{prop:GGGR-mult}. Let $u \in \mathcal{O}^F$ be a class representative and assume either $A_{\bG}(u)$ is abelian or $\A_{\bG}(u)$ is non-abelian and $A_{\bG}(u)^F = A_{\bG}(u)$ then
\begin{equation}\label{eq:div-criterion}
\epsilon_{\chi}\cdot\chi(u) = \frac{1}{|A_{\bG}(u)^F|}\sum_{\iota} \zeta_{\iota}' q^{b_{\iota}}\psi_{\iota}(1)^2,
\end{equation}
where the sum is taken over all $\iota \in \mathcal{N}_{\bG}^F$ with $\mathcal{O}_{\iota} = \mathcal{O}$. In particular the expression on the right hand side is an algebraic integer.
\end{cor}

\begin{proof}
From \cite[Corollary 2.6]{geck:1999:character-sheaves-and-GGGRs} we get
\begin{align*}
\chi(u_x) &= \frac{1}{|A_{\bG}(u)|}\sum_{\iota} \zeta_{\iota}' q^{b_{\iota}}Y_{\iota}(u_x)\langle \chi, D_{\bG}(\gamma_{\iota}) \rangle,\\
&= \epsilon_{\chi}\cdot \frac{1}{|Z_{\bG}(u)||A_{\btG}(\tilde{u})|}\sum_{\iota} \zeta_{\iota}' q^{b_{\iota}}Y_{\iota}(u_x)\langle \chi^*, \gamma_{\iota} \rangle,\\
&= \epsilon_{\chi} \cdot \frac{|Z_{\bG}(u)|}{|Z_{\bG}(u)^F||Z_{\bG}(u)||A_{\btG}(\tilde{u})|}\sum_{\iota} \zeta_{\iota}' q^{b_{\iota}}Y_{\iota}(u_x)\overline{Y_{\iota}(u_x)},\\
&= \epsilon_{\chi} \cdot \frac{1}{|Z_{\bG}(u)^F||A_{\btG}(\tilde{u})|}\sum_{\iota} \zeta_{\iota}' q^{b_{\iota}}Y_{\iota}(u_x)\overline{Y_{\iota}(u_x)},\\
&= \epsilon_{\chi} \cdot \frac{1}{|A_{\bG}(u)^F|}\sum_{\iota} \zeta_{\iota}' q^{b_{\iota}}Y_{\iota}(u_x)\overline{Y_{\iota}(u_x)}.
\end{align*}
Assume first that $A_{\bG}(u)$ is abelian and recall $Y_{\iota}(u_x) = \tilde{\psi}_{\iota}(xF)$. As $\tilde{\psi}_{\iota}$ will be a linear character every character value will be a root of unity, in particular $Y_{\iota}(u_x)\overline{Y_{\iota}(u_x)} = 1 = \tilde{\psi}_{\iota}(1) = \psi_{\iota}(1)^2$.

We assume now that $A_{\bG}(u)$ is non-abelian and $A_{\bG}(u)^F = A_{\bG}(u)$. Using the arguments in \cite[Proposition 3.1]{geck:1999:character-sheaves-and-GGGRs} we see the condition $n_{\chi} = |A_{\bG}(u)|$ implies $A_{\bG}(u_x)^F = A_{\bG}(u_x)$. Let $g \in \bG$ be such that $u_x = gug^{-1}$, (i.e.\ $g^{-1}F(g) \mapsto x \in A_{\bG}(u)$), then the component group has the form $A_{\bG}(u_x) = \{gyC_{\bG}(u)^{\circ}g^{-1} \mid y \in A_{\bG}(u)\}$. If every element of $A_{\bG}(u_x)$ is fixed by $F$ then
\begin{equation*}
F(gy)C_{\bG}(u)^{\circ}F(g)^{-1} = gyC_{\bG}(u)^{\circ}g^{-1} \Rightarrow (g^{-1}F(g))yC_{\bG}(u)^{\circ}(g^{-1}F(g))^{-1} = yC_{\bG}(u)^{\circ}
\end{equation*}
for all $y \in A_{\bG}(u)$. The image of $g^{-1}F(g)$ in $A_{\bG}(u)$ is $x$, which implies that $x$ is in the centre of $A_{\bG}(u)$. This ensures $\psi_{\iota}(x) = \psi_{\iota}(1)$ because $x$ must be in the kernel of every character so again $Y_{\iota}(u_x)\overline{Y_{\iota}(u_x)} = \psi_{\iota}(x)^2 = \psi_{\iota}(1)^2$.
\end{proof}

This corollary is the key ingredient for our argument. It provides a divisibility criterion for the fourth root of unity $\zeta_{\iota}'$ in the ring of algebraic integers, which will lead to a divisibility criterion in $\mathbb{Z}$.

\section{The Special Orthogonal and Symplectic Groups}\label{sec:orth-symp-groups}
We want to adapt the argument of \cite[Theorem 3.8]{geck:1999:character-sheaves-and-GGGRs} to symplectic groups and even dimensional special orthogonal groups. Recall that we assume $p$ is a good prime, in particular $p \neq 2$. Our computation of the fourth root of unity will use a result of Digne--Lehrer--Michel using Gauss sums which depends upon two choices, the first being a choice of primitive fourth root of unity in $\overline{\mathbb{Q}}_{\ell}$. Following \cite[\S 1.B]{bonnafe:2006:sln} we fix an injective homomorphism of groups $\varphi : \mathbb{Q/Z} \to \overline{\mathbb{Q}}_{\ell}^{\times}$ and denote by $\tilde{\varphi} : \mathbb{Q} \to \overline{\mathbb{Q}}_{\ell}^{\times}$ the composition of $\varphi$ with the natural quotient map $\mathbb{Q} \to \mathbb{Q/Z}$; we have $\Ker(\tilde{\varphi}) = \mathbb{Z}$. We now define $j$ to be $\tilde{\varphi}(1/4)$, which is a primitive fourth root of unity in $\overline{\mathbb{Q}}_{\ell}$.

The second choice we need to make is of a square root of $p$ in $\overline{\mathbb{Q}}_{\ell}$, which we will now do following \cite[\S 36]{bonnafe:2006:sln}. We fix a non-trivial additive character $\chi_1 : \mathbb{F}_p \to \overline{\mathbb{Q}}_{\ell}^{\times}$ and denote by $\chi_s : \mathbb{F}_{p^s} \to \overline{\mathbb{Q}}_{\ell}^{\times}$ the additive character $\chi_s = \chi_1 \circ \Tr_s$ where $\Tr_s : \mathbb{F}_{p^s} \to \mathbb{F}_p$ is the field trace. We denote by $\theta_s : \mathbb{F}_{p^s}^{\times} \to \overline{\mathbb{Q}}_{\ell}^{\times}$ the unique linear character of degree 2 and define the associated Gauss sum to be
\begin{equation*}
\mathcal{G}_s(\theta_s) = \sum_{x \in \mathbb{F}_{p^s}^{\times}} \theta_s(x)\chi_s(x)
\end{equation*}
We denote by $p^{\frac{1}{2}}$ our fixed square root of $p$ in $\overline{\mathbb{Q}}_{\ell}$, which is chosen in the following way
\begin{equation*}
p^{\frac{1}{2}} = \begin{cases}
\mathcal{G}_1(\theta_1) &\text{ if }p\equiv 1 \pmod{4},\\
j^{-1}\mathcal{G}_1(\theta_1) &\text{ if }p\equiv 3 \pmod{4}.
\end{cases}
\end{equation*}
Now given any $a \in \mathbb{Z}$ we denote by $p^{\frac{a}{2}}$ the term $(p^{\frac{1}{2}})^a$.

\begin{table}[t]
\centering
\begin{tabular}{>{$}l<{$}>{$}l<{$}>{$}l<{$}>{$}l<{$}}
\toprule\addlinespace
\multicolumn{1}{c}{Group} & \multicolumn{1}{c}{Condition} & \multicolumn{1}{c}{Partition} & \multicolumn{1}{c}{$|A_{\bG}(u)|$}\\\addlinespace
\midrule\addlinespace
\Sp_{2n} & n = 1+\cdots+k & (2,4,6,\dots,2k) & 2^k\\\addlinespace
\SO_{2n} & n = 2k^2 & (1,3,5,\dots,4k-1) & 2^{2k-1}\\\addlinespace
\bottomrule
\end{tabular}
\caption{Conditions for the existence of cuspidal pairs.}
\label{tab:cond-cuspidal}
\end{table}

In \cref{tab:cond-cuspidal} we have listed the conditions for the existence of a cuspidal pair in special orthogonal and symplectic groups. This information has been adapted from \cite{lusztig:1984:intersection-cohomology-complexes} and we see that in any given case there is only one cuspidal pair. The argument we will employ is based on induction and for this to work we will need to know $\zeta_{\iota_0}'$ where $\iota_0$ is the unique cuspidal pair of $\bG = \SL_2(\overline{\mathbb{F}}_p)$ and $G = \SL_2(q)$. However this can be deduced from a result of Digne, Lehrer and Michel.

\begin{lem}[Digne--Lehrer--Michel]\label{lem:dlm-SL2}
Assume $\bG = \SL_2(\overline{\mathbb{F}}_p)$ where $p \neq 2$ and $G = \SL_2(q)$ with $q = p^a$. Let $\iota_0$ be the unique cuspidal pair of $\bG$ then
\begin{equation*}
\zeta_{\iota_0}' = \begin{cases}
(-1)^a &\text{if }p \equiv 1 \pmod{4},\\
(-j)^a &\text{if }p \equiv -1 \pmod{4}.
\end{cases}
\end{equation*}
\end{lem}

\begin{proof}
Taking $n = e = 2$ in \cite[Proposition 2.8]{digne-lehrer-michel:1997:gelfand-grave-characters-disconnected} and using \cite[36.3]{bonnafe:2006:sln} for the computation of the Gauss sum $\mathcal{G}_s(\theta_s)$ we have
\begin{equation*}
\zeta_{\iota_0} = \begin{cases}
(-1)^{a-1} &\text{if }p\equiv 1 \pmod{4},\\
(-1)^{a-1}j^a &\text{if }p\equiv -1 \pmod{4}.\\
\end{cases}
\end{equation*}
As $\delta_{\iota_0} = -1$ the statement clearly follows by the definition of $\zeta_{\iota_0}'$.
\end{proof}

\begin{rem}
It should be noted that the statement of \cite[Proposition 2.8]{digne-lehrer-michel:1997:gelfand-grave-characters-disconnected} depends upon the validity of \ref{ax:1}. However as the full character table of $\SL_2(q)$ is known, this could be computed without this result.
\end{rem}

We now turn to the final piece of information we will need regarding the symplectic groups, which involves a simple counting argument. Let $\iota_0 \in \mathcal{N}_{\bG}$ be a cuspidal pair then we define $X_{\iota_0} = \{\iota \in \mathcal{N}_{\bG}^F \mid \mathcal{O}_{\iota} = \mathcal{O}_{\iota_0}\}$.

\begin{lem}\label{lem:sp-odd-even}
Let $\bG = \Sp_{2n}(\overline{\mathbb{F}}_p)$ and assume $n = 1 + 2 + \dots + k$, for some $k \geqslant 1$. Under this assumption there exists a unique cuspidal pair $\iota_0$ in $\mathcal{N}_{\bG}$. In this situation we have
\begin{equation*}
|\{\iota \in X_{\iota_0} \mid m_{\iota}\text{ is even}\}| = |\{\iota \in X_{\iota_0} \mid m_{\iota}\text{ is odd}\}| = 2^{k-1},
\end{equation*}
where $m_\iota = \rank(\bL_{\iota}/Z(\bL_{\iota})) = \rank(\bL_{\iota}) - \dim(Z(\bL_{\iota})^{\circ})$.
\end{lem}

\begin{proof}
By \cite[11.6.1]{lusztig:1984:intersection-cohomology-complexes} the elements of $X_{\iota_0}$ are parameterised by certain unordered pairs $(A,B)$, where $A$ is a finite subset of $\mathbb{Z}_{\geqslant 0}$ and $B$ is a finite subset of $\mathbb{Z}_{\geqslant 1}$ such that $|A| + |B|$ is odd. Specifically we have the following, which follows as a consequence of \cite[Corollary 12.4]{lusztig:1984:intersection-cohomology-complexes}.

Let $d_{\iota_0} = -k$ if $k$ is odd and $k+1$ if $k$ is even then $d_{\iota_0}$ is an odd number satisfying $n = \frac{1}{2}d_{\iota_0}(d_{\iota_0}-1)$. If $k$ is even we have a bijection $\iota \mapsto (A_{\iota},B_{\iota})$ between $X_{\iota_0}$ and the set of all symbols $(A_{\iota},B_{\iota})$ satisfying
\begin{enumerate}
	\item $A_{\iota} \subseteq \{0,2,4,\dots,2d_{\iota_0}-2\}$,
	\item $B_{\iota} \subseteq \{2,4,\dots,2d_{\iota_0}-2\}$,
	\item $A_{\iota} \cap B_{\iota} = \emptyset$ and $A_{\iota} \cup B_{\iota} = \{0,2,4,\dots,2d_{\iota_0}-2\}$.
\end{enumerate}
In particular this bijection is such that $\iota_0 \mapsto (\{0,2,4\dots,2d_{\iota_0}-2\},\emptyset)$. If $k$ is odd we again have a bijection $\iota \mapsto (A_{\iota},B_{\iota})$ between $X_{\iota_0}$ and the set of all symbols $(A_{\iota},B_{\iota})$ satisfying
\begin{enumerate}
	\item $A_{\iota}, B_{\iota} \subseteq \{1,3,5,\dots,1-2(d_{\iota_0}+1)\}$,
	\item $A_{\iota} \cap B_{\iota} = \emptyset$ and $A_{\iota} \cup B_{\iota} = \{1,3,5,\dots,1-2(d_{\iota_0}+1)\}$.
\end{enumerate}
In particular this bijection is such $\iota_0 \mapsto (\emptyset,\{1,3,5,\dots,1-2(d_{\iota_0}+1)\})$. We define the defect of a pair $(A_{\iota},B_{\iota})$ to be the value $d_{\iota} := |A_{\iota}| - |B_{\iota}|$. In both cases we have $(A_{\iota_0},B_{\iota_0})$ is the unique symbol of defect $d_{\iota_0}$.

The defect of a symbol is important as it is related to the value $m_{\iota}$ in the following way. For each $\iota \in X_{\iota_0}$ let $k_{\iota}$ be $d_{\iota}-1$ if $d_{\iota} \geqslant 1$ and $-d_{\iota}$ if $d_{\iota} \leqslant -1$, note that $k_{\iota_0} = k$. The rank $m_{\iota}$ is then given by
\begin{equation*}
m_{\iota} = \frac{1}{2}k_{\iota}(k_{\iota}+1).
\end{equation*}
We can now consider the number of pairs $\iota \in X_{\iota_0}$ such that $m_{\iota}$ is even and the number of pairs such that $m_{\iota}$ is odd. Assume $k_{\iota} = 2y_{\iota}$ is even then it is not difficult to see that $m_{\iota} \equiv y_{\iota} \pmod{2}$. Similarly if $k_{\iota} = 2y_{\iota} + 1$ then $m_{\iota} \equiv y_{\iota} + 1 \pmod{2}$.

Assume $k=2y$ is even then for each $0 \leqslant x \leqslant k$ there are precisely $\binom{k}{x}$ subsets $B_{\iota} \subseteq \{2,4,\dots,2d_{\iota_0}-2\}$ such that $|B_{\iota}| = x$. If $x \leqslant y$ then the symbols $(A_{\iota},B_{\iota})$ associated to these subsets all have defect $d_{\iota} = k+1 - 2x = 2(y-x) + 1 \geqslant 1$. Similarly if $x > y$ then the symbols have defect $d_{\iota} = 2x - (k-1) = 2(x-y) + 1 \leqslant -1$. By the above we always have $m_{\iota} \equiv y-x \pmod{2}$. The result then follows from the following basic fact of binomial coefficients
\begin{equation}\label{eq:binom-fact}
\sum_{x\text{ even}}\binom{k}{x} = \sum_{x\text{ odd}}\binom{k}{x} = \sum_{z = 0}^{k-1} \binom{k-1}{z} = 2^{k-1}.
\end{equation}
Assume $k = 2y + 1$ is odd then for each $0 \leqslant x \leqslant k$ there are precisely $\binom{k}{x}$ subsets $A_{\iota} \subseteq \{1,3,\dots,1-2(d_{\iota_0}+1)\}$ such that $|A_{\iota}| = x$. The symbols $(A_{\iota},B_{\iota})$ associated to these subsets all have defect $d_{\iota} = 2x-k = 2(x-y) -1$. By the above we always have $m_{\iota} \equiv y-x+1 \pmod{2}$. The result again follows from \eqref{eq:binom-fact}.
\end{proof}

\begin{rem}
It should be noted that some of the statements of \cite[Corollary 12.4]{lusztig:1984:intersection-cohomology-complexes} are not quite correct. In particular the description of the symbol corresponding to the cuspidal pair in the case $k$ is odd. This was corrected by Shoji in \cite[Remark 5.8]{shoji:1997:unipotent-characters-of-finite-classical-groups} and we have used the corrected statement above. Finally we know \cite[Corollary 12.4(c)]{lusztig:1984:intersection-cohomology-complexes} is correct by the remark at the end of \cite[\S 2.B]{geck-malle:2000:existence-of-a-unipotent-support}, where the work of Shoji is also referenced. This is all we have used from this reference.
\end{rem}

We are now in a state where we may prove the final result of this paper. The following fourth roots of unity were computed by Waldspurger for symplectic / special orthogonal groups in \cite[\S V.8 - Proposition]{waldspurger:2001:integrales-orbitales-nilpotentes}. Here we merely give an alternative proof for these values, (note that one easily checks that the expressions coincide).

\begin{thm}
Assume $\bG_n$ is $\Sp_{2n}(\overline{\mathbb{F}}_p)$ or $\SO_{2n}(\overline{\mathbb{F}}_p)$ and, if such a pair exists, let $\iota_0$ be the unique cuspidal pair in $\mathcal{N}_{\bG_n}$. Let $\epsilon \in \{\pm 1\}$ be such that $p \equiv \epsilon \pmod{4}$, then we have
\begin{equation*}
\zeta_{\iota_0}' =
\begin{cases}
\epsilon^{\frac{an}{2}} &\text{ if }n\text{ is even},\\
(-1)^{an} &\text{ if }n\text{ is odd and }\epsilon=1,\\
(-j)^{an} &\text{ if }n\text{ is odd and }\epsilon=-1.
\end{cases}
\end{equation*}
\end{thm}

\begin{proof}
The proof of this statement is an adaptation of the proof of \cite[Theorem 3.8]{geck:1999:character-sheaves-and-GGGRs}, which is a proof by induction on $n$. The proof is identical for the case of special orthogonal groups but requires slightly more work for the case of symplectic groups. Note we will let $k$ be the appropriate value as described in \cref{tab:cond-cuspidal}. If $n=0$ then $\bG_n$ is a torus and the formula is trivially true. If $n=1$ then $\bG_n = \Sp_2(\overline{\mathbb{F}}_p) \cong \SL_2(\overline{\mathbb{F}}_p)$ and we can see that the formula coincides with that given by \cref{lem:dlm-SL2}. If $n=2$ then $\bG_n = \SO_4(\overline{\mathbb{F}}_p)$ and the simply connected covering $\bG_{\simc}$ of $\bG_n$ is isomorphic to $\SL_2(\overline{\mathbb{F}}_p) \times \SL_2(\overline{\mathbb{F}}_p)$. By \cite[\S 10.1]{lusztig:1984:intersection-cohomology-complexes} $\iota_0$ is the image of the direct product $\upsilon_0 \times \upsilon_0$, where $\upsilon_0$ is the unique cuspidal pair of $\SL_2(\overline{\mathbb{F}}_p)$. There are two possible rational structures on $\bG_{\simc}$, either $\bG_{\simc}^F = \SL_2(q) \times \SL_2(q)$ or $\bG_{\simc}^F = \SL_2(q^2)$. However, in both cases we have $\zeta_{\iota_0}' = (\zeta_{\upsilon_0}')^2$ and by \cref{lem:dlm-SL2} we can see the formula is valid.

Now we assume that $n\geqslant 3$ and the statement is true for all $\bG_m$ with $m < n$. If $\iota \in X_{\iota_0}$ then by \cite[\S 10.4 and \S 10.6]{lusztig:1984:intersection-cohomology-complexes} we have $\bL_{\iota}/Z(\bL_{\iota})^{\circ}$ is isomorphic to $\bG_{m_{\iota}}$ for some $m_{\iota} \leqslant n$. Let $N$ be the number of positive roots of $\bG$ then as $\rank \bG = \rank \bL_{\iota}$ for all $\iota$ we can express $b_{\iota}$ in the following way
\begin{align*}
b_{\iota} &= \frac{1}{2}(\dim \bG - \dim \mathcal{O}_{\iota} - \dim Z(\bL_{\iota})^{\circ}),\\
&= \frac{1}{2}(2N + \rank \bG - \dim \mathcal{O}_{\iota} - \dim Z(\bL_{\iota})^{\circ}),\\
&= \frac{1}{2}(2N - \dim \mathcal{O}_{\iota}) + \frac{1}{2}(\rank \bL_{\iota} - \dim Z(\bL_{\iota})^{\circ}),\\
&= \dim \mathfrak{B}_u + \frac{m_{\iota}}{2},
\end{align*}
where we have used the dimension formula given in \cite[Theorem 5.10.1]{carter:1993:finite-groups-of-lie-type} to obtain the last equality.

We now consider the criterion given to us from \cref{cor:4th-root-diviser}. Recall that $u$ is a well-chosen element of a symplectic or special orthogonal group, so in particular $A_{\bG}(u)^F = A_{\bG}(u)$. Furthermore the component group $A_{\bG}(u)$ is abelian which means $\psi_{\iota}(1) = 1$ for all $\iota$. Using the information we have gathered we may rewrite the sum in \cref{cor:4th-root-diviser} in the following way
\begin{align*}
\sum_{\iota} \zeta_{\iota}' q^{b_{\iota}}\psi_{\iota}(1)^2 = \zeta_{\iota_0}'q^{\dim \mathfrak{B}_u + \frac{n}{2}} + \sum_{\iota\neq \iota_0} \zeta_{\iota}' q^{\dim \mathfrak{B}_u + \frac{m_{\iota}}{2}} = q^{\dim \mathfrak{B}_u}\left(\zeta_{\iota_0}'q^{\frac{n}{2}} + \sum_{\iota\neq \iota_0} \zeta_{\iota}' q^{\frac{m_{\iota}}{2}} \right).
\end{align*}
The statement of \cref{cor:4th-root-diviser} says that $|A_{\bG}(u)|$ divides the expression on the right hand side in the ring of algebraic integers. The order of the component group is a power of $2$, therefore as $q$ is a power of an odd prime we must have
\begin{equation}\label{eq:sp-4th-roots-1}
|A_{\bG}(u)|\quad\text{divides}\quad\zeta_{\iota_0}'p^{\frac{an}{2}} + \sum_{\iota \in X_{\iota_0}\setminus \{\iota_0\}} \zeta_{\iota}' p^{\frac{am_{\iota}}{2}}
\end{equation}
in the ring of algebraic integers. As $n \geqslant 3$ we have $|A_{\bG}(u)| \geqslant 4$ so 4 divides the expression on the right which means we may consider the expression modulo 4. However we must be careful only to reduce integer powers of $p$ modulo 4 and never rational powers.

We wish to consider the possibilities for the terms in the sum modulo 4. Let us assume $\iota \in X_{\iota_0} \setminus \{\iota_0\}$ then reducing only integer powers of $p$ modulo 4 we see
\begin{equation*}
\zeta_{\iota}'p^{\frac{am_i}{2}} \equiv \begin{cases}
1 \pmod{4} &\text{if }m_{\iota}\text{ is even},\\
(-1)^ap^{\frac{a}{2}} \pmod{4} &\text{if }m_{\iota}\text{ is odd and }\epsilon=1,\\
(-j)^ap^{\frac{a}{2}} \pmod{4} &\text{if }m_{\iota}\text{ is odd and }\epsilon=-1.
\end{cases}
\end{equation*}
Let us assume $\bG_n = \SO_{2n}(\overline{\mathbb{F}}_p)$. From \cref{tab:cond-cuspidal} we see $m_{\iota}$ is even for all $\iota \in X_{\iota_0}$ then \eqref{eq:sp-4th-roots-1} tells us
\begin{equation}\label{eq:so-4th-roots-1}
4\quad\text{divides}\quad\zeta_{\iota_0}'p^{\frac{an}{2}} - 1
\end{equation}
in the ring of algebraic integers. Multiplying the term on the right by the same term with $-$ exchanged by $+$ we have
\begin{equation*}
4\quad\text{divides}\quad(\zeta_{\iota_0}')^2p^{an} - 1.
\end{equation*}
All the terms on the right are integers so we must have 4 divides this term in $\mathbb{Z}$. In other words $(\zeta_{\iota_0}')^2p^{an} - 1 \equiv 0 \pmod{4} \Rightarrow (\zeta_{\iota_0}')^2 \equiv 1 \pmod{4}$ as $n$ is even, hence $\zeta_{\iota_0}' = \pm 1$. Returning to \eqref{eq:so-4th-roots-1} we see that all the values in the expression are integers, so it must be true that 4 divides this term in $\mathbb{Z}$. In particular $\zeta_{\iota_0}'\epsilon^{\frac{an}{2}} \equiv 1 \pmod{4}$, which provides the desired formula.

Now assume $\bG_n = \Sp_{2n}(\overline{\mathbb{F}}_p)$. Using \cref{lem:sp-odd-even} we have the sum in \eqref{eq:sp-4th-roots-1} can be expressed as
\begin{equation}\label{eq:sp-4th-roots-2}
\zeta_{\iota_0}'p^{\frac{an}{2}} + \sum_{\iota \in X_{\iota_0}\setminus \{\iota_0\}} \zeta_{\iota}' p^{\frac{am_{\iota}}{2}} = \begin{cases}
\zeta_{\iota_0}'p^{\frac{an}{2}} + (2^{k-1}-1) +2^{k-1}\eta^a p^{\frac{a}{2}} &\text{if }n\text{ is even},\\
\zeta_{\iota_0}'p^{\frac{an}{2}} + 2^{k-1} + (2^{k-1}-1)\eta^a p^{\frac{a}{2}} &\text{if }n\text{ is odd}.
\end{cases}
\end{equation}
where $\eta = -1$ or $-j$ depending upon the congruence of $p$ modulo 4. We would like to eliminate the rational powers of $p$. Assume for the moment that $k > 2$ then 4 divides $2^{k-1}$, so we may simplify extensively the expressions in \eqref{eq:sp-4th-roots-2}. Let us consider the following products
\begin{equation*}
(\zeta_{\iota_0}'p^{\frac{an}{2}} - 1)(\zeta_{\iota_0}'p^{\frac{an}{2}} + 1) \qquad\text{and}\qquad (\zeta_{\iota_0}'p^{\frac{an}{2}} - \eta^ap^{\frac{a}{2}})(\zeta_{\iota_0}'p^{\frac{an}{2}} + \eta^ap^{\frac{a}{2}})
\end{equation*}
where the left comes from the case $n$ is even and the right comes from the case $n$ is odd. By \eqref{eq:sp-4th-roots-1} and \eqref{eq:sp-4th-roots-2} we have 4 divides these products in the ring of algebraic integers. Expanding the brackets we see all terms in the expression are in $\mathbb{Z}$, hence 4 must divide these expressions in $\mathbb{Z}$. Reducing modulo 4 we get $(\zeta_{\iota_0}')^2p^{an} - 1 \equiv 0 \pmod{4}$, (because $\eta^2 = \epsilon$ and $\epsilon^ap^a \equiv 1 \pmod{4}$), in particular $(\zeta_{\iota_0}')^2 \equiv \epsilon^{an} \pmod{4}$.

If $n$ is even then the argument is identical to the above argument for $\SO_{2n}(\overline{\mathbb{F}}_p)$. Therefore let us assume $n$ is odd, then $(\zeta_{\iota_0}')^2 \equiv \epsilon^a \pmod{4}$ hence $\zeta_{\iota_0} = \pm\eta^a$ where $\eta$ is as in \eqref{eq:sp-4th-roots-2}. Returning to \eqref{eq:sp-4th-roots-1} and \eqref{eq:sp-4th-roots-2} we see
\begin{equation*}
4\quad\text{divides}\quad \zeta_{\iota_0}'p^{\frac{an}{2}} - \eta^ap^{\frac{a}{2}}
\end{equation*}
in the ring of algebraic integers. Assume $n = 2x + 1$ then we have $p^{\frac{an}{2}} = p^{ax}p^{\frac{a}{2}}\equiv \epsilon^{ax}p^{\frac{a}{2}} \pmod{4}$. This implies
\begin{equation*}
4\quad\text{divides}\quad (\zeta_{\iota_0}'\epsilon^{ax} - \eta^a)p^{\frac{a}{2}}
\end{equation*}
in the ring of algebraic integers. Assume for a contradiction that $\zeta_{\iota_0}' = -\epsilon^{ax}\eta^a$ then this says 4 divides $-2\eta^ap^{\frac{a}{2}}$ in the ring of algebraic integers. In particular there exists an algebraic integer $y$ such that $2\eta^ap^{\frac{a}{2}} = 4y \Rightarrow \eta^ap^{\frac{a}{2}}=2y$. Squaring this expression we see $\eta^{2a}p^a = 4y^2$ but $\eta^{2a}p^a \in \mathbb{Z}$ and as 2 divides the right hand side we must have 2 divides $p^a$ in $\mathbb{Z}$ but this is impossible as $p$ is odd. Therefore $\zeta_{\iota_0}' = \epsilon^{ax}\eta^a$ and the result follows by noticing $\eta^{an} = \eta^{2ax}\eta^a = \epsilon^{ax}\eta^a$.

To finish the proof we must deal with the case where $k=2$, in other words $n=3$ and $\bG_n = \Sp_6(\overline{\mathbb{F}}_p)$. In this case $2^{k-1} = 2$, so returning to \eqref{eq:sp-4th-roots-2} we see
\begin{equation}\label{eq:sp6-4th-roots}
\zeta_{\iota_0}'p^{\frac{an}{2}} + \sum_{\iota \in X_{\iota_0}\setminus \{\iota_0\}} \zeta_{\iota}' p^{\frac{am_{\iota}}{2}} = \zeta_{\iota_0}'p^{\frac{an}{2}} + 2 + \eta^a p^{\frac{a}{2}}.
\end{equation}
Again we consider the product
\begin{equation*}
(\zeta_{\iota_0}'p^{\frac{an}{2}} + 2 + \eta^a p^{\frac{a}{2}})(\zeta_{\iota_0}'p^{\frac{an}{2}} - 2 - \eta^a p^{\frac{a}{2}}) = (\zeta_{\iota_0}')^2p^{an} - 4 - 4\eta^ap^{\frac{a}{2}} - \eta^{2a}p^a.
\end{equation*}
Reducing this modulo 4 we see $(\zeta_{\iota_0}')^2 \equiv \epsilon^a \pmod{4}$, hence $\zeta_{\iota_0} = \pm\eta^a$ as before. Again we write $n = 2x+1$ then returning to \eqref{eq:sp6-4th-roots} and reducing modulo 4 we have
\begin{equation*}
4\qquad\text{divides}\qquad 2 + (\zeta_{\iota_0}'\epsilon^{ax} + \eta^a) p^{\frac{a}{2}}
\end{equation*}
in the ring of algebraic integers. If $\zeta_{\iota_0}' = -\epsilon^{ax}\eta^a$ then the above statement says 4 divides 2 in the ring of integers, a contradiction. Hence we must have $\zeta_{\iota_0}' = \epsilon^{ax}\eta^a$ and as above this is $\eta^{an}$ as required.
\end{proof}

One would hope that the above argument may work for other simple groups with a disconnected centre, unfortunately this is not the case. For example when $\bG$ is a spin group the degree of $\psi_{\iota_0}$ is a power of 2, where $\iota_0$ is the cuspidal pair not coming from the generalised Springer correspondence of the corresponding special orthogonal group. In particular when we reduce modulo 4 in the above argument the term containing the unknown fourth root of unity will become zero.

\bibliographystyle{amsalpha-var}
\small
\begin{spacing}{0.96}
\bibliography{bibliography}

\providecommand{\bysame}{\leavevmode\hbox to3em{\hrulefill}\thinspace}
\providecommand{\MR}{\relax\ifhmode\unskip\space\fi MR }
\providecommand{\MRhref}[2]{%
  \href{http://www.ams.org/mathscinet-getitem?mr=#1}{#2}
}
\providecommand{\href}[2]{#2}
\begin{thebibliography}{Tay12b}

\bibitem[Bon06]{bonnafe:2006:sln}
C.~Bonnaf{{\'e}}, \emph{Sur les caract{\`e}res des groupes r{\'e}ductifs finis
  {\`a} centre non connexe: applications aux groupes sp{\'e}ciaux lin{\'e}aires
  et unitaires}, Ast{\'e}risque (2006), no.~306, vi+165.

\bibitem[Car93]{carter:1993:finite-groups-of-lie-type}
R.~W. Carter, \emph{Finite groups of {L}ie type}, Wiley Classics Library, John
  Wiley \& Sons Ltd., Chichester, 1993, Conjugacy classes and complex
  characters, Reprint of the 1985 original, A Wiley-Interscience Publication.

\bibitem[DLM97]{digne-lehrer-michel:1997:gelfand-grave-characters-disconnected}
F.~Digne, G.~I. Lehrer, and J.~Michel, \emph{On {G}el'fand-{G}raev characters
  of reductive groups with disconnected centre}, J. Reine Angew. Math.
  \textbf{491} (1997), 131--147.

\bibitem[DM91]{digne-michel:1991:representations-of-finite-groups-of-lie-type}
F.~Digne and J.~Michel, \emph{Representations of finite groups of {L}ie type},
  London Mathematical Society Student Texts, vol.~21, Cambridge University
  Press, Cambridge, 1991.

\bibitem[Gec93]{geck:1993:basic-sets-II}
M.~Geck, \emph{Basic sets of {B}rauer characters of finite groups of {L}ie
  type. {II}}, J. London Math. Soc. (2) \textbf{47} (1993), no.~2, 255--268.

\bibitem[Gec96]{geck:1996:on-the-average-values}
\bysame, \emph{On the average values of the irreducible characters of finite
  groups of {L}ie type on geometric unipotent classes}, Doc. Math. \textbf{1}
  (1996), No. 15, 293--317 (electronic).

\bibitem[Gec99]{geck:1999:character-sheaves-and-GGGRs}
\bysame, \emph{Character sheaves and generalized {G}elfand-{G}raev characters},
  Proc. London Math. Soc. (3) \textbf{78} (1999), no.~1, 139--166.

\bibitem[Gec03]{geck:2003:intro-to-algebraic-geometry}
\bysame, \emph{An introduction to algebraic geometry and algebraic groups},
  Oxford Graduate Texts in Mathematics, vol.~10, Oxford University Press,
  Oxford, 2003.

\bibitem[Gec04]{geck:2004:on-the-schur-indices}
\bysame, \emph{On the {S}chur indices of cuspidal unipotent characters}, Finite
  groups 2003, Walter de Gruyter GmbH \& Co. KG, Berlin, 2004, pp.~87--104.

\bibitem[GH08]{geck-hezard:2008:unipotent-support}
M.~Geck and D.~H{{\'e}}zard, \emph{On the unipotent support of character
  sheaves}, Osaka J. Math. \textbf{45} (2008), no.~3, 819--831.

\bibitem[GM00]{geck-malle:2000:existence-of-a-unipotent-support}
M.~Geck and G.~Malle, \emph{On the existence of a unipotent support for the
  irreducible characters of a finite group of {L}ie type}, Trans. Amer. Math.
  Soc. \textbf{352} (2000), no.~1, 429--456.

\bibitem[H{\'{e}}z04]{hezard:2004:thesis}
D.~H{\'{e}}zard, \emph{{S}ur le support unipotent des
  faisceaux-caract\`{e}res}, Ph.D. thesis, Universit\`{e} Lyon 1, 2004,
  \url{http://tel.archives-ouvertes.fr/tel-00012071}.

\bibitem[Kaw85]{kawanaka:1985:GGGRs-and-ennola-duality}
N.~Kawanaka, \emph{Generalized {G}elfand-{G}raev representations and {E}nnola
  duality}, Algebraic groups and related topics ({K}yoto/{N}agoya, 1983), Adv.
  Stud. Pure Math., vol.~6, North-Holland, Amsterdam, 1985, pp.~175--206.

\bibitem[Kaw86]{kawanaka:1986:GGGRs-exceptional}
\bysame, \emph{Generalized {G}elfand-{G}raev representations of exceptional
  simple algebraic groups over a finite field. {I}}, Invent. Math. \textbf{84}
  (1986), no.~3, 575--616.

\bibitem[Lus84a]{lusztig:1984:characters-of-reductive-groups}
G.~Lusztig, \emph{Characters of reductive groups over a finite field}, Annals
  of Mathematics Studies, vol. 107, Princeton University Press, Princeton, NJ,
  1984.

\bibitem[Lus84b]{lusztig:1984:intersection-cohomology-complexes}
\bysame, \emph{Intersection cohomology complexes on a reductive group}, Invent.
  Math. \textbf{75} (1984), no.~2, 205--272.

\bibitem[Lus86]{lusztig:1986:on-the-character-values}
\bysame, \emph{On the character values of finite {C}hevalley groups at
  unipotent elements}, J. Algebra \textbf{104} (1986), no.~1, 146--194.

\bibitem[Lus88]{lusztig:1988:reductive-groups-with-a-disconnected-centre}
\bysame, \emph{On the representations of reductive groups with disconnected
  centre}, Ast{\'e}risque (1988), no.~168, 10, 157--166, Orbites unipotentes et
  repr{{\'e}}sentations, I.

\bibitem[Lus92]{lusztig:1992:a-unipotent-support}
\bysame, \emph{A unipotent support for irreducible representations}, Adv. Math.
  \textbf{94} (1992), no.~2, 139--179.

\bibitem[Lus08]{lusztig:2008:irreducible-representations-of-finite-spin-groups}
\bysame, \emph{Irreducible representations of finite spin groups}, Represent.
  Theory \textbf{12} (2008), 1--36.

\bibitem[Lus09]{lusztig:2009:unipotent-classes-and-special-Weyl}
\bysame, \emph{Unipotent classes and special {W}eyl group representations}, J.
  Algebra \textbf{321} (2009), no.~11, 3418--3449.

\bibitem[Miz80]{mizuno:1980:conjugacy-classes-of-E7-and-E8}
K.~Mizuno, \emph{The conjugate classes of unipotent elements of the {C}hevalley
  groups {$E_{7}$} and {$E_{8}$}}, Tokyo J. Math. \textbf{3} (1980), no.~2,
  391--461.

\bibitem[Sho87]{shoji:1987:green-functions-of-reductive-groups}
T.~Shoji, \emph{Green functions of reductive groups over a finite field}, The
  {A}rcata {C}onference on {R}epresentations of {F}inite {G}roups ({A}rcata,
  {C}alif., 1986), Proc. Sympos. Pure Math., vol.~47, Amer. Math. Soc.,
  Providence, RI, 1987, pp.~289--301.

\bibitem[Sho97]{shoji:1997:unipotent-characters-of-finite-classical-groups}
\bysame, \emph{Unipotent characters of finite classical groups}, Finite
  reductive groups ({L}uminy, 1994), Progr. Math., vol. 141, Birkh{\"a}user
  Boston, Boston, MA, 1997, pp.~373--413.

\bibitem[SS70]{springer-steinberg:1970:conjugacy-classes}
T.~A. Springer and R.~Steinberg, \emph{Conjugacy classes}, Seminar on
  {A}lgebraic {G}roups and {R}elated {F}inite {G}roups ({T}he {I}nstitute for
  {A}dvanced {S}tudy, {P}rinceton, {N}.{J}., 1968/69), Lecture Notes in
  Mathematics, Vol. 131, Springer, Berlin, 1970, pp.~167--266.

\bibitem[Tay12a]{taylor:2012:finding-characters-satisfying}
J.~Taylor, \emph{Finding characters satisfying a maximal condition for their
  unipotent support},
  \texttt{arXiv:\href{http://arxiv.org/abs/1211.2551}{1211.2551}}
  \texttt{[math.rt]}, (2012).

\bibitem[Tay12b]{taylor:2012:thesis}
\bysame, \emph{On unipotent supports of reductive groups with a disconnected
  centre}, Ph.D. thesis, University of Aberdeen, 2012,
  \url{http://tel.archives-ouvertes.fr/tel-00709051}.

\bibitem[Wal01]{waldspurger:2001:integrales-orbitales-nilpotentes}
J.-L. Waldspurger, \emph{Int{\'e}grales orbitales nilpotentes et endoscopie
  pour les groupes classiques non ramifi{\'e}s}, Ast{\'e}risque (2001),
  no.~269, vi+449.

\bibitem[Wal63]{wall:1963:conjugacy-classes-classical-groups}
G.~E. Wall, \emph{On the conjugacy classes in the unitary, symplectic and
  orthogonal groups}, J. Austral. Math. Soc. \textbf{3} (1963), 1--62.

\end{thebibliography}
\end{spacing}
\end{document}